\documentclass{article}
\usepackage[ansinew]{inputenc}
\usepackage{amsfonts, color}
\usepackage{amssymb,amsmath}
\usepackage[french, english]{babel}
\usepackage[T1]{fontenc}
\usepackage{textcomp}
\usepackage{latexsym}
\usepackage{graphicx}
\usepackage{amsthm}
\usepackage{mathrsfs}

\usepackage{ifthen}
\usepackage{array}

\newcommand	{\QQ}{\mathbb{Q}}

\newcommand	{\XX}{\mathbb{X}}
\newcommand	{\NN}{\mathbb{N}}

\newcommand	{\DD}{\mathbb{D}}

\newcommand	{\RR}{\mathbb{R}}
\newcommand	{\PP}{\mathbb{P}}
\newcommand	{\EE}[1][NoArg]{\mathbb{E}\ifthenelse{\equal{#1}{NoArg}}{}{\left[#1\right]}}

\newtheorem	{thm}		{Theorem}[section]
\newtheorem	{lem} 	[thm]	{Lemma}

\newtheorem     {rem}           {Remark}
\newtheorem	{prop}	[thm]{Proposition}
\newtheorem	{cor}		[thm]{Corollary}

\usepackage{braket}
\usepackage{hyperref}

\newtheorem*{theorem8}{Theorem \ref{thm:geneticvariance}}
\usepackage{bbm}

\newcommand{\dd}{\mathrm{d}}

\makeatletter
\def\thickhline{%
  \noalign{\ifnum0=`}\fi\hrule \@height \thickarrayrulewidth \futurelet
   \reserved@a\@xthickhline}
\def\@xthickhline{\ifx\reserved@a\thickhline
               \vskip\doublerulesep
               \vskip-\thickarrayrulewidth
             \fi
      \ifnum0=`{\fi}}
\makeatother

\newlength{\thickarrayrulewidth}
\title{The gene's-eye view of quantitative genetics.}
\author{Philibert Courau, Amaury Lambert, Emmanuel Schertzer}

\definecolor{darkgreen}{RGB}{34,139,34}

\begin{document}
\hypersetup{pageanchor=false}
\begin{titlepage}
\maketitle

\begin{abstract}
Modelling the evolution of a continuous trait in a biological population is one of the oldest problems in evolutionary biology, which led to the birth of quantitative genetics. With the recent development of GWAS methods, it has become essential to link the evolution of the trait distribution to the underlying evolution of allelic frequencies at many loci, co-contributing to the trait value. The way most articles go about this is to make assumptions on the trait distribution, and use Wright's formula to model how the evolution of the trait translates on each individual locus. Here, we take a gene's-eye view of the system, starting from an explicit finite-loci model with selection, drift, recombination and mutation, in which the trait value is a direct product of the genome. We let the number of loci go to infinity under the assumption of strong recombination, and characterize the limit behavior of a given locus with a McKean-Vlasov SDE and the corresponding Fokker-Planck IPDE. In words, the selection on a typical locus depends on the mean behaviour of the other loci which can be approximated with the law of the focal locus. Results include the independence of two loci and explicit stationary distribution for allelic frequencies at a given locus (under some assumptions on the fitness function). 
\end{abstract}
\end{titlepage}
\hypersetup{pageanchor=true}
\pagenumbering{arabic}
\tableofcontents

\section{Introduction}
The purpose of this article is to present polygenic adaptation from the gene's-eye view. Polygenic adaptation refers to the response of a biological population to natural selection on a trait, which is co-determined by a large number of loci \cite{hollinger2019polygenic}. Understanding polygenic adaptation is of great importance because most biological traits relevant to fitness are complex, and therefore this form of adaptation is expected to be extremely common, in particular in short timescales. It is also of particular interest if one is to use the result of Genome-Wide Association Studies (GWAS) to detect traces of past selection \cite{Simons2018}.

The study of polygenic adaptation has a very rich history, deeply rooted into the field of quantitative genetics. The foundational paper which enabled this field was Fisher's 1918 article \cite{fisher1918correlation}. In this paper, Fisher convincingly argued that Mendelian heredity of many factors co-determining a trait was compatible with the observed laws for heredity of continuous traits. It had been known since Galton's study on human height \cite{galton1886regression} that traits tend to be normally distributed within a population. Fisher justified this from a Mendelian perspective, arguing that if the trait results from the sum of contributions from many close-to-independent loci, the trait will be normally distributed in a well-mixed population \cite{Barton2017}. Furthermore, the heritable part of the trait variance can be expressed with the sum of the variances of allelic frequencies at many loci, thus offering a compelling link between macroscopic and microscopic parameters. This view of quasi-independent genes with additive effects would prove essential in the development of models to predict the long-term evolution of the population \cite{walsh2018evolution}.

One convenient feature of the infinitesimal model is that it makes it possible to model the evolution of a trait without taking into consideration its genetic architecture. For instance, Lande's 1976 article \cite{lande1976natural} famously modelled the evolution of a normally-distributed trait under Gaussian stabilizing selection and drift, under the assumption that the genetic variance remains constant through time. Since then, many models have explored partial differential equations (PDEs) that rely on the same assumptions concerning the distribution of offspring traits relative to their parents' (see Section 1.2.3 of \cite{dekens2022quantifying} for a summary of the state of the art). We call this approach the trait's eye-view.

In this article we present an alternative approach, which we call the gene's-eye view of quantitative genetics. We will study a Stochastic Differential Equation (SDE) to represent the evolution of a large panmictic haploid biological population with $L$ biallelic loci under the forces of recombination, mutation, natural selection and genetic drift.

\subsection{Definition of the model.\label{sec:defmodel}}
To motivate the definition of our model let us first consider a single locus/$\{-1,+1\}$ alleles. When the population $N$ is large, under the joint action of mutation, selection and genetic drift, the evolution of the frequency of $+1$-allele is well approximated by the Wright-Fisher diffusion 
\begin{equation}
\label{eq:WF-11}
\dd X_t = s(X_t)X_t(1-X_t)\dd t + \overline{\Theta}(X_t)\dd t + \sqrt{X_t(1-X_t)} \dd B_t
\end{equation}
where $s$ is a frequency dependent selection term, $B$ is a Browian motion and 
$$
\overline{\Theta}(x) \ := \ \theta^{(+)}(1-x) -\theta^{(-)} x
$$
with $\theta^{(+)},\theta^{(-)}$ the rates of mutations from the $-1$ allele to the $+1$ allele and back.

\bigskip

The aim of the present article is to  consider a $L$-loci/$\{-1,+1\}$ allele model. In the following,
we will consider a diffusion analog to the classical Wright-Fisher diffusion (\ref{eq:WF-11}), but in order to provide more intuition on our continuum model, we first 
consider  a discrete population in discrete time comprised of $N$ individuals, each with $L$ genes encoded by an element of the hypercube $\square_{[L]}:=\{-1,+1\}^L$. 

The evolution of the population results from the combined effect of Selection (S), Recombination (R), Mutation (M) and random sampling. Those fundamental evolutionary forces are encoded by the following parameters.
\begin{enumerate}
\item[(S)] Let $W \ : \ \square_{[L]} \to \RR$.
\item[(R)] Let $\rho>0$ and $\nu$ be a probability measure on $\{{\mathcal I} \subset [L] : \mathcal{I} \neq \emptyset, [L]\}$.
\item[(M)] Let $\theta^{(+)},\theta^{(-)}\geq0$.
\end{enumerate}
Reproduction then occurs according to Wright-Fisher sampling. Start with $N$ genomes at time $k=0$, meaning a vector $(g_1,\dots,g_N)\in (\square_{[L]})^N$. At every generation $k>1$, each of the $N$ genomes comprising the new generation independently picks two parent genomes $\gamma_1,\gamma_2$ with probability proportional to their fitnesses \begin{equation}
    \exp\left(\frac{L}{N}W(\gamma_1)\right), \  \exp\left(\frac{L}{N}W(\gamma_2)\right) \label{eq:WLN}
\end{equation}
The function $W$ is often referred to as the log-fitness function. As we will see, the $\frac{L}{N}$ factor is important so that the strength of selection felt by one locus is of order 1.

With probability $1-\rho/N$, the new genome $g_0$ inherits the whole genome of one of its two parents $g_1$ and $g_2$ chosen uniformly at random. With probability $\rho/N$, a recombination occurs and $g_0$ inherits a mixture of the genetic material of the two parents according to $\nu$. More precisely, a subset ${\mathcal I} \subset [L]$ is sampled according to $\nu$ and $g_0$ inherits the genetic material of $g_1$ at loci ${\mathcal I}$, and the material of $g_2$ at loci ${\mathcal I}^{c}$. 

Finally, we assume
that each  generation, each locus on each allele can mutate with probability $\theta^{(+)}/N$ (resp., $\theta_{2}/N$)  from $-1$ to $+1$ (resp. $+1$ to $-1$) after reproduction.

\bigskip

Let $k\in\NN$, and define
${\bf X}^{(N)}_k = (X^{(N)}_k(\gamma))_{\gamma\in\square_{[L]}}$ 
where $X^{(N)}_k(\gamma)$
is the frequency of genotype $\gamma$ at generation $k$. The process $({\bf X}^{(N)}_k)_{k\in\NN}$ is valued in the space  of probability distributions on the hypercube $\square_{[L]}$ that we denote by $\XX^{[L]}$.
Let us now consider the large population limit ($N\to+\infty$) of the rescaled process $({\bf X}^{(N)}_{\lfloor tN\rfloor})_{t\geq0}$. A straightforward generator computation indicates that the process converges to a diffusive limit $({\bf X}_t)_{t\geq0}$ valued in $\XX^{[L]}$ which is solution to a Stochastic Differential Equation (SDE) on $\XX^{[L]}$ of the form
\begin{equation}\label{eq:intromaster}
 \dd\mathbf{X}_t = (\rho R(\mathbf{X}_t)  + \Theta(\mathbf{X}_t) + LS(\mathbf{X}_t)) \dd t
 + \Sigma(\mathbf{X}_t)\dd\mathbf{B}_t
\end{equation}
that we now explain. This equation will be the focus of this article.

\bigskip

{\bf Recombination.} 
For a subset $\mathcal{I}\subseteq [L]$ and ${\bf x}\in \XX_L$, define $\mathbf{x}^{\mathcal{I}}$ the marginal of $\mathbf{x}$ on the hypercube $\square_{\mathcal{I}}:=\{-1,+1\}^{\mathcal{I}}$. Let $\mathbf{x}^{\mathcal{I}}\otimes \mathbf{x}^{\mathcal{I}^c}$ be the product measure on $\square_{[L]}$ of $\mathbf{x}^{\mathcal{I}}$ and $\mathbf{x}^{\mathcal{I}^c}$.
The recombinator operator has been extensively studied  in the deterministic setting, see e.g., \cite{rabani1995computational,MARTINEZ2017115,Baakebaake,BaakeMut,entropy_production,evans2013mutation} and is defined as
\begin{equation}
R : \left\{
\begin{array}{ll}
 \XX^{[L]} & \to \RR^{{\square}_{[L]}}  \\ 
 {\bf x}   & \mapsto \sum\limits_{\emptyset\subsetneq \mathcal{I}\subsetneq [L]} \nu(\mathcal{I})(\mathbf{x}^{\mathcal{I}}\otimes \mathbf{x}^{\mathcal{I}^c} - \mathbf{x} )\label{def:recombinator}
\end{array}
\right.
\end{equation}
Note that up to replacing $\nu$ with $\tilde{\nu}:\mathcal{I}\mapsto \frac{\nu(\mathcal{I}) + \nu(\mathcal{I}^c)}{2}$, we can and will assume that for any $\mathcal{I}\subseteq [L]$, $\nu(\mathcal{I})=\nu(\mathcal{I}^c)$.

\bigskip

{\bf Mutation.}
The mutator is defined as 
\begin{equation*}
\Theta: \left\{
\begin{array}{ll}
    \XX^{[L]} &\longrightarrow \RR^{\square_{[L]}}  \\
    \mathbf{x} &\longmapsto |\theta|\sum\limits_{\ell\in [L]} 
    \left(\mathbf{x}^{[L]\smallsetminus\{\ell\}}\otimes \mathcal{L}_\theta -\mathbf{x}\right)
\end{array}
\right.
\end{equation*}
where $|\theta|$ is the total mutation rate per locus and $\mathcal{L}_\theta$ is the mutational law defined with
\begin{align*}
|\theta| :=& \theta^{(+)}+\theta^{(-)}    &
\mathcal{L}_\theta :=& \frac{\theta^{(-)}}{|\theta|}\delta_{-1}+\frac{\theta^{(+)}}{|\theta|}\delta_{+1}.
\end{align*}

\bigskip

{\bf Selection.}
The operator $S : \ \square_{[L]} \ \to \ \RR^{\square_{[L]}}$ is the selector defined with 
$$
S(\mathbf{x})(\gamma) :=
x(\gamma)(W(\gamma)-\mathbf{x}[W(g)])
 = \mathbf{Cov}_{\mathbf{x}}\left[W(g),\mathbbm{1}_{[g = \gamma]}\right]
$$
where $\mathbf{x}[\cdot]$ and $\mathbf{Cov}_{\mathbf{x}}[\cdot,\cdot]$ are the expectation and the covariance function for a random genotype $g$ with law $\mathbf{x}$. This can be thought of as an application of the Price equation \cite{PRICE1970} to the trait $F^\gamma(g):=\mathbbm{1}_{[g=\gamma]}$. See for instance \cite{LESSARD1997119}. The factor $L$ in front of $S$ in (\ref{eq:intromaster}) corresponds to the strength of selection, which stems from (\ref{eq:WLN}) and will be discussed in the next subsection.

\bigskip

{\bf Genetic Drift.} 
The stochastic term is the traditional multiallele Wright-Fisher diffusion term \cite{steinrucken2013explicit} and with xthe following covariance structure. 
$$
 \left<(\Sigma(\mathbf{X}_t)\dd\mathbf{B}_t)(\gamma),(\Sigma(\mathbf{X}_t)\dd\mathbf{B}_t)(\gamma')\right> = \delta_{\gamma,\gamma'} X_{t}(\gamma) -  X_{t}(\gamma)X_{t}(\gamma').
$$
More precisely, we will consider
${\bf B}\equiv (B_t(\gamma_1,\gamma_2))_{t\in[0,T];\gamma_1,\gamma_2\in\square_{[L]}}$ a Gaussian process indexed by $\square_{[L]}\times \square_{[L]}$ such that $B(\gamma_1,\gamma_2) = -B(\gamma_2,\gamma_1)$, and $B(\gamma_1,\gamma_2),B(\gamma_3,\gamma_4)$ are independent Brownian motions if $(\gamma_1,\gamma_2)\notin \{(\gamma_3,\gamma_4),(\gamma_4,\gamma_3)\}$.

Finally, let  $\mathcal{M}(\square_{[L]}\times \square_{[L]},\RR^{\square_{[L]}})$ denote the space of linear functions from $\square_{[L]}\times\square_{[L]}$ to $\RR^{\square_{[L]}}$. Then 
$$
\Sigma :  \ \XX^{[L]}  \to \   \mathcal{M}\left(\square_{[L]}\times \square_{[L]},\RR^{\square_{[L]}}\right)
$$
is defined such that 
\begin{equation}\label{def:drift}
\forall \gamma\in\square_{[L]}, \ \ \ \     (\Sigma(\mathbf{X}_t)\dd\mathbf{B}_t)(\gamma) :=
\sum\limits_{\hat{\gamma}\neq \gamma} \sqrt{X_t(\gamma) X_t(\hat{\gamma})}\dd B_t(\gamma,\hat{\gamma})
\end{equation}

\begin{rem}Existence and uniqueness of solutions to (\ref{eq:intromaster}) was obtained from the martingale problem in \cite{ethier}.
\end{rem}

\subsection{Propagation of chaos.} For any $\ell\in[L]$, define 
$$
p^\ell(\mathbf{X}_t) := \sum\limits_{\gamma\in\square_{[L]}} \mathbbm{1}_{[\gamma^\ell = +1]} X_t(\gamma)
$$ as the frequency of $+1$ allele at locus $\ell$ at time $t$.   We are interested in describing the joint evolution of the $p^\ell(X_t)$'s
together with the allelic averaged process $\mu_{\mathbf{X}_t}$, where for $\mathbf{x}\in \XX^{[L]}$
$$
\mu_{\mathbf{x}} := \frac{1}{L}\sum\limits_{\ell\in [L]} \delta_{p^\ell(\mathbf{x})}
$$
When focusing on just a few loci, the high dimensionality of the SDE (\ref{eq:WF-11}) renders the problem practically intractable. However, as $L\to+\infty$ and recombination becomes sufficiently strong to neglect correlations between loci, a mean field approximation applies where any focal locus only experiences the averaged effect of the rest of the genome, a phenomenon commonly known as the propagation of chaos. See Fig \ref{fig:propagation_chaos}.

\bigskip

Let us now describe our results in more details.
We will consider a sequence of models indexed by $L$ where the dependence in $L$ is in $$
(\rho, \nu, W, {\bf X}_0) \equiv (\rho^L,\nu^L, W^L, {\bf X}_{0}^L),$$ whereas other parameters remain constant.  Our theorem is concerned with the limit of (\ref{eq:intromaster}) as $L\to +\infty$. 
We make several assumptions on the order of the parameters.

\bigskip

The most restrictive assumption is that the log-fitness $W$
is a quadratic polynomial with bounded coefficients. More precisely,
    \begin{equation}
\forall \gamma\in\square_{[L]}, \ \ W(\gamma) := U(Z(\gamma))
\qquad;\qquad
Z(\gamma) := \frac{1}{L}\sum\limits_{\ell\in [L]} \gamma^\ell
    \label{eq:defquadraW}
    \end{equation}
where $Z(\gamma)$ is the additive trait value associated with genotype $\gamma$ and $U$ is a polynomial of order 1 or 2. In particular, if $U$ is of order 2 it can be written up to an additive constant
\begin{equation}\label{eq:U-square}
U(z) = -\kappa (z-z^*)^2
\end{equation}
for some $\kappa, z^*\in\RR$.
This assumption on the functional form for the log-fitness is classic in quantitative genetics. Take $z^*\in[-1,1]$.
The case $\kappa>0$ (resp., $\kappa<0$)
corresponds to a scenario of stabilizing (resp., disruptive selection) on a quantitative additive trait \cite{lande1976natural,Simons2018}. The assumption that the fitness within the population is mostly determined by stabilizing selection on a (multi-dimensional) highly polygenic trait is known as Fisher's geometric fitness model and has had countless applications \cite{tenaillon2014utility}. 

Let us briefly discuss the scaling. With the functional form of the log-fitness, selection acts on the additive trait $Z$. Under sufficiently strong recombination, it is reasonable to assume that selection impacts all loci similarly. Due to the additive nature of the trait, this influence is evenly distributed across all loci. Therefore, if we start with a selection of order $L$ at the trait level as in (\ref{eq:intromaster}), we expect selection to have an effect of order $1$ at the locus level. In particular, the rescaling in (\ref{eq:WLN}) ensures that the dynamics at the gene level remain well-defined and non-degenerate. 

 We will need to assume that recombination is strong as compared to selection. In order to quantify this relation, we first define some summary statistics related to the recombination measure $\nu$. For $A\subseteq [L]$, let $\nu^A$ be the marginal of $\nu$ on $A$. 
Define the recombination rate between two distinct loci $\ell_1,\ell_2\in [L]$
$$
r_{\{\ell_1,\ell_2\}} := \;\nu^{\{\ell_1,\ell_2\}}(\{\ell_1\}) + \nu^{\{\ell_1,\ell_2\}}(\{\ell_2\})
$$
We assume that $\nu$ is non-degenerate, that is, $\inf_{\ell_1\neq\ell_2} r_{\{\ell_1,\ell_2\}}>0$.
For a given locus $\ell_0\in[L]$, define 
$r_{\ell_0}^*$ the harmonic recombination at $\ell_0$ as
\begin{equation}
    \label{eq:defrell}
\frac{1}{r_{\ell_0}^*} :=  \frac{1}{L-1}
\sum\limits_{\ell_1\in [L]\smallsetminus\{\ell_0\}} \frac{1}{r_{\{\ell_0,\ell_1\}}}
\end{equation} 
and the average harmonic recombination rate along the genome as
\begin{align}
\label{eq:defr1}
\frac{1}{r^{**}} :=&
\frac{1}{L}\sum\limits_{\ell_0\in [L]} \frac{1}{r_{\ell_0}^*}
\end{align}

The next theorem states that under strong enough recombination (see conditions (\ref{ass:rhobeta1}-\ref{ass:rhoell1ell2})), a mean-field approximation applies and the $p^\ell(X_t)$'s converge to independent  McKean-Vlasov diffusions \cite{chaintron2021propagation1}
\begin{eqnarray}\label{eq:McKean-limitW}
 \dd f_t \ = \ \overline{s}(\mathscr{L}(f_t)) \ f_t(1-f_t) \dd t +
    \overline{\Theta}(f_t)\dd t
    +\sqrt{f_t(1-f_t)}\;\dd B_t  \\
    \mbox{where} \ \  \overline{s}(\zeta)\;=2U'(<\zeta,2\mbox{Id}-1>)\qquad;\qquad
\overline{\Theta}(x) := \theta^{(+)}(1-x) - \theta^{(-)} x \nonumber
\end{eqnarray}
and $\mathscr{L}(f_t)$ denotes the law of the process $f_t$. This shows that in polygenic adapation, selection at the genome level ``percolates'' at the locus level  where it dictates a non-linear Wright-Fisher dynamics.  Further, the averaged process $\mu_{{\bf X}_t}$  converges to $\mathscr{L}(f_t)$ which in turn can be calculated as the weak solution of the non-linear IPDE (integro partial differential equation) corresponding to the Fokker-Planck equation associated to (\ref{eq:McKean-limitW})
\begin{equation}\label{eq:McKean-limitPDE}
\partial_t u_t(x) =-\partial_{x}\left[ \left(\bar s\left(u_t(\cdot)\right)x(1-x) +\overline{\Theta}(x) \right)u_t(x)\right] +  \frac{1}{2}  \partial_{xx} \left(x(1-x)u_t(x)\right).
\end{equation}
\begin{figure}
    \centering
    \includegraphics[width=.8\textwidth]{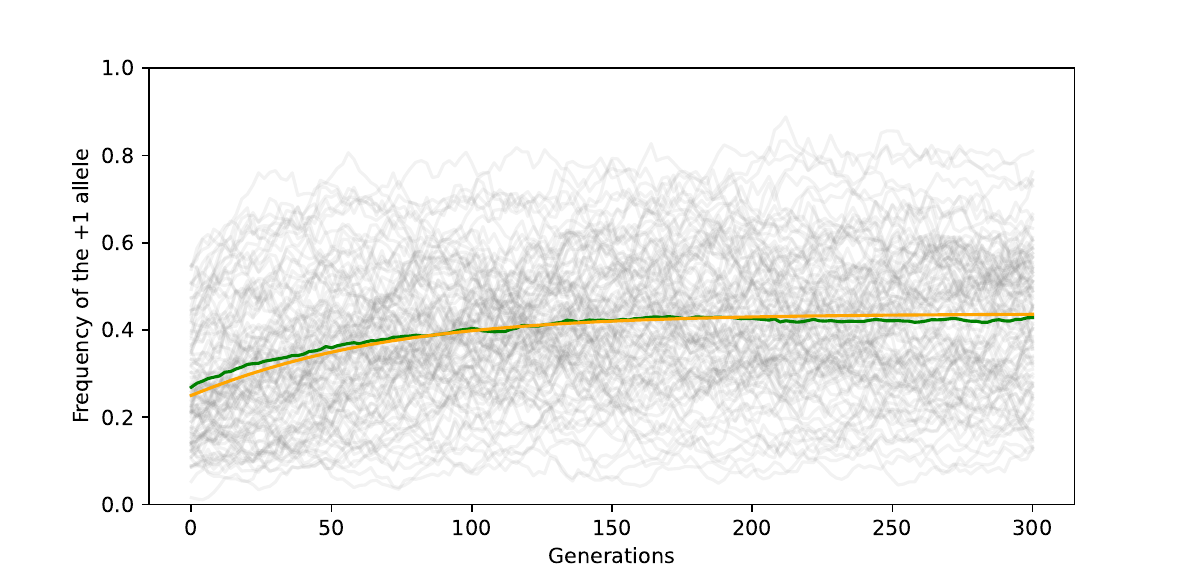}
    \caption{We consider the discete population of $N=1000$ individuals with $L=100$ genes each for $T = 1000$ generations, simulated as detailed in Section \ref{sec:defmodel} with single uniform crossing-over (see below). The mutation rates are $\theta=(1.1,3.3)$, the strength of stabilizing in (\ref{eq:U-square}) is $\kappa =15$ and $z^*=0$. At time $t=0$, the population is distributed according to the neutral discrete Wright-Fisher equilibrium with mutation rate $\theta$. The grey lines show the trajectories of the frequency of the +1 allele at each individual locus. The green line is the average of the grey lines. The orange line corresponds to the mean-field approximation  (\ref{eq:Lande}), computed with an Euler approximation scheme.  The code is available on  \url{https://github.com/PhCourau/Gene-s_eye_view_of_quantitative_genetics}.
 }
    \label{fig:propagation_chaos}
\end{figure}

We now state our main result on the convergence to (\ref{eq:McKean-limitPDE}). We let $\mathcal{P}([0,1])$ be the set of probability measures on $[0,1]$ equipped with the weak topology.
\begin{thm}\label{thm:strongcvgrho}
Assume that $\mu_{\mathbf{X}_0}$ converges in law to a deterministic measure $m_0$, and that
\begin{equation}\label{ass:rhobeta1}
\rho r^{**} \gg L^{2}\ln(\rho) 
\end{equation}
Then
\begin{enumerate}
    \item For every $T>0$
    $$(\mu_{\mathbf{X}_t})_{t\in [0,T]} \quad\Longrightarrow \quad \left(\mathscr{L}(f_t)\right)_{t\in [0,T]} $$
    where $\Longrightarrow$ denotes weak convergence in $\DD([0,T],\mathcal{P}([0,1]))$ for the Skorokhod J1 topology (see \cite{billingsley2013convergence}, Chapter 3), and $f_t$ is the unique solution of the McKean-Vlasov equation (\ref{eq:McKean-limitW}) with initial distribution $m_0$. In particular, $(\mathscr{L}(f_t))_{t\in[0,T]}$ is the unique weak solution to (\ref{eq:McKean-limitPDE}) (see Section \ref{sec:Ltoinf} for details).
    \item Let $n\in\NN$. Assume there exists a sequence integers $\ell_1^L<\dots<\ell_n^L$ in $[L]$, such that  $(p^{\ell_i^L}(\mathbf{X}_0))_{i\in [n]}$ has law converging to $\mathfrak{P}_0\in \mathcal{P}([0,1]^n)$ and that
\begin{align}
        \min_{i\in [n]}\rho r_{\ell_i^L}^* \gg L^{2}\ln(\rho)        \label{ass:rhorell}\\
        \min_{\substack{i,j\in [n]\\i\neq j}} \rho r_{\{\ell_{i}^L,\ell_{j}^L\}} \gg L
        \label{ass:rhoell1ell2}
\end{align}
    Then for every $T>0$
    \begin{itemize}
        \item $(p^{\ell_i^L}(\mathbf{X}_t))_{t\in [0,T];i\in [n]}$ converges in distribution to $n$ diffusions $(\overline{p}^i_t)_{t\in [0,T];i\in [n]}$ solutions to
        \begin{equation}
            \dd\overline{p}_t^i  = \overline{s}(\mathscr{L}(f_t))\overline{p}_t^i(1-\overline{p}_t^i) \dd t +
            \overline{\Theta}(\overline{p}_t^i)\dd t
            +\sqrt{\overline{p}_t^i(1-\overline{p}_t^i)}\;\dd B_t^i 
            \label{eq:realMcKean}    
        \end{equation}
        with $(B^i)_{i\in [n]}$ independent Brownian motions and initial conditions $\mathscr{L}((\overline{p}_0^i)_{i\in [n]}) =\mathfrak{P}_0$.
   \item 
    If $\mathfrak{P}_0 = m_0^{\otimes n}$ then $(p^{\ell_i^L}(\mathbf{X}_t))_{t\in [0,T];i\in [n]}$ converges in distribution to $n$ independent McKean-Vlasov diffusions (\ref{eq:McKean-limitW}).
    \end{itemize}
\end{enumerate}
\end{thm}
Let us briefly discuss the previous assumptions. First, the theorem relies on the parameter $r^{**}$, the importance of which was already noted by Bulmer \cite{Bulmer_1974}. 
Secondly, the strong recombination conditions (\ref{ass:rhobeta1}--\ref{ass:rhoell1ell2}) are satisfied provided that recombination grows significantly with \(L\). 

The intuition is as follows: recombination needs to be strong enough to sufficiently break correlations between loci, allowing a mean-field approximation to be valid. The technical challenge arises because selection tends to induce correlations along the genome. For instance, if the optimum \(z^*\) is at 0 and one knows that \(\gamma^\ell = +1\) at a given locus, selection will tend to favor \(-1\) at other loci to compensate, keeping the trait near the optimum. In this way, selection introduces negative correlations across the genome.

Since the strength of selection scales with \(L\), it becomes necessary to assume a recombination rate that also scales sufficiently-specifically, one that is large in \(L\)-to counteract these correlations.

To provide more intuition on the strong recombination conditions (\ref{ass:rhobeta1}-\ref{ass:rhoell1ell2}) of the previous theorem, we explore some of the classical recombination models.
\begin{itemize}
   \item The model known to population geneticists as free recombination  corresponds to the case where $\nu(\mathcal{I}) = \frac{1}{2^L}$ for any $\mathcal{I}\subseteq [L]$. In this case we can check that
$$ \forall \ell_0 \in [L], \ \  r_{\ell_0}^* = r^{**} = \frac{1}{2}$$
so that conditions the strong recombination conditions (\ref{ass:rhobeta1}-\ref{ass:rhoell1ell2}) are satisfied provided that $\rho \gg L^2\ln(L)$. 
    \item \textbf{Single crossing-over.} 
    Let $\mu$ be a probability measure 
    on $[0,1]$ with stricty positive continuous density. Then $\nu$ is the law of the random set 
    $$
    {\mathcal J} = \left\{ i\in[L], \frac{i}{L+1} \leq X  \right\}, \ \ \mbox{where $\mathscr{L}(X) = \mu$.}
    $$
    The strong recombination conditions (\ref{ass:rhobeta1}-\ref{ass:rhoell1ell2}) are satisfied provided that $\rho \gg L^2\ln(L)^2$.
    \item{\bf Multiple crossing-overs.}
   Consider a Poisson Point Process with an intensity measure with a strictly positive continuous density, seen as a random set of points $\lambda_1<\dots<\lambda_{N}$. We add the boundary points $\lambda_0:=0$ and $\lambda_{N+1}:=1$.
    Then $\nu$ is the law of the random set 
    $$
    {\mathcal J} = \left\{i \in[L] : \exists k\leq  \frac{N+1}{2} \ \ \mbox{s.t.}  \ \  \frac{i}{L+1} \in [\lambda_{2k},\lambda_{2k+1})\right\}
    $$
\end{itemize}

\bigskip

\subsection{Invariant distribution(s).} 
Assume that $\theta^{(+)},\theta^{(-)}>0$. For any $y\in\RR$, define 
\begin{equation}\label{def:Piy}
\Pi_y(x) = C_y x^{2\theta^{(+)}-1} (1-x)^{2\theta^{(-)}-1} e^{2xy}\dd x
\end{equation}
with $C_y$ a normalization constant so that $\Pi_y$ is a probability on $[0,1]$. It is well known that the classical Wright-Fisher SDE (\ref{eq:WF-11}) with a constant selection term $s$ has a unique invariant distribution $\Pi_s$. From there, the Lipschitzness of $\overline{s}$ trivially implies
\begin{thm}\label{thm:invariant}
Define
\begin{equation}\label{eq:tmpchi}
\chi: y  \longmapsto \overline{s}(\Pi_y)
\end{equation}
where $\bar s$ is the function defined in the McKean-Vlasov equation (\ref{eq:McKean-limitW}).
Then 
$\chi$ admits at least one fixed point. Futher, 
\begin{itemize}
\item  The set of invariant distribution for  (\ref{eq:McKean-limitW}) coincides with
$$
\{\Pi_{y^*}\  : \ \chi(y^*)=y^* \}
$$
\item Assume that the initial condition of (\ref{eq:McKean-limitW}) is given by $\Pi_{y^*}$ with $\chi(y^*)=y^*$. Then $(f_t)_{t\geq0}$ is distributed as a classical Wright-Fisher diffusion (\ref{eq:WF-11}), with a constant selection term $s \equiv \overline{s}(\Pi_{y^*})$ and initial distribution $\Pi_{y^*}$.
\end{itemize}
\end{thm}

We emphasize that there may exist several solutions  to the fixed point problem of Theorem \ref{thm:invariant}. For $U$ as in (\ref{eq:U-square}) with $z^*=0$  and $\theta^{(+)}=\theta^{(-)}$, we show in  Corollary \ref{cor:statio} the existence of a critical $\kappa_c<0$ at which the system undergoes a pitchfork-like bifurcation. See Fig \ref{fig:bifurcation}.

\begin{figure}[ht]
    \centering
    \includegraphics[width=.6\textwidth]{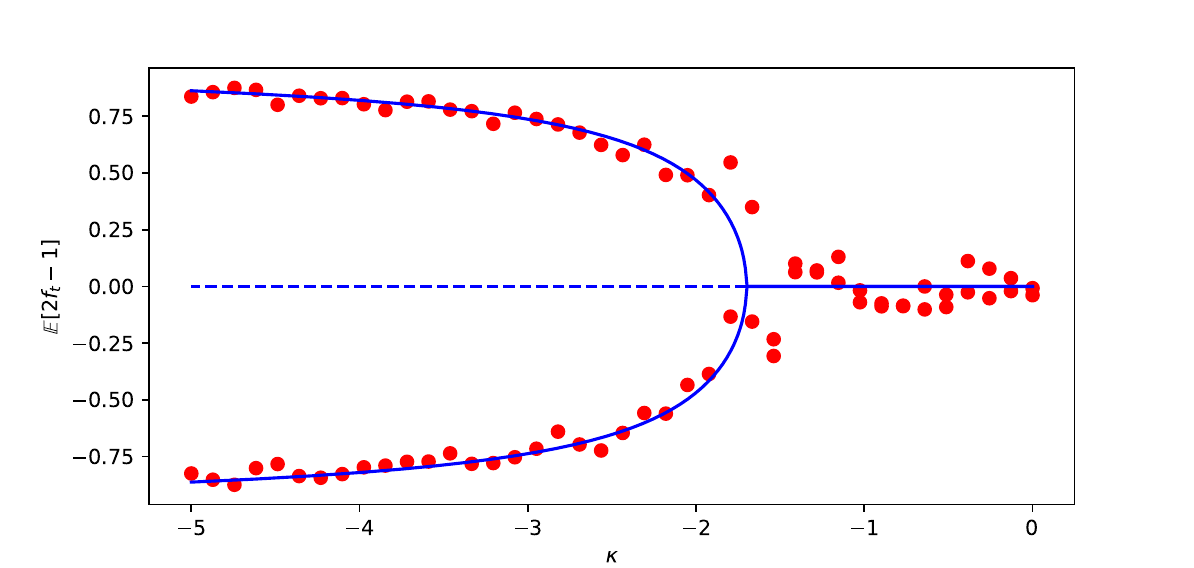}
    \caption{Supercritical pitchfork bifurcation in disruptive selection (\ref{eq:U-square}). For each value of $\kappa$ on the $x$-axis, we simulated two discrete population (as detailed in Section \ref{sec:defmodel}) with $N=200,L=100$ after $T=20N$ generations, with initial conditions "all $+1$" or "all $-1$". The red dots correspond to  $<\mu_{\mathbf{X}_T},2\mbox{Id}-1>$ at the end of the simulation.  The mutation rates are fixed $\theta^{(+)}=\theta^{(-)}=0.6$, and the selection optimum is $z^*=0$. The blue lines correspond to the possible values of $\EE[2f_t-1]$ for stationary solutions to the limit equation (\ref{eq:McKean-limitW}). Corollary \ref{cor:statio} predicts a pitchfork-like bifurcation at $\kappa_c=-1.7$.}
    \label{fig:bifurcation}
\end{figure}

\bigskip

\subsection{Trait Variance. \label{sec:traitvar}}
A biologically important quantity is the  variance of the trait $\mathbf{Var}_{\mathbf{X}_t}[Z(g)]$. Our mean-field approximation in Theorem \ref{thm:strongcvgrho}(1) entails that the variance goes to $0$. The next result provides a second order approximation.
\begin{thm}\label{thm:geneticvariance}
Assume that the assumptions of Theorem \ref{thm:strongcvgrho} part 1. hold. 
    Set
\begin{equation}\label{def:epsilon}
\varepsilon_L:=\frac{1}{\sqrt{\rho r^{**}}}
\end{equation}
Define the genetic variance $\sigma_t^2:=4\EE[f_t(1-f_t)]$. Then
    $$\EE\left[\sup\limits_{t\in [\varepsilon_L,T]}\left|L\mathbf{Var}_{\mathbf{X}_t}[Z(g)] - \sigma_t^2\right|\right]
    \longrightarrow 0$$
\end{thm}

Equation (\ref{eq:McKean-limitPDE}) has a biologically important corollary. Assume for simplicity there is no mutation $\theta^{(+)}=\theta^{(-)}=0$. The mean of the trait distribution satisfies
\begin{equation}\label{eq:Lande}
\frac{\dd}{\dd t}\EE[2f_t-1] = U'(\EE[2f_t-1])\times \sigma_t^2
\end{equation}
This is known as Lande's equation \cite{Hayward2022}. The term $U'(\EE[2f_t-1])$ is called the selection gradient.

\subsection{Significance and extensions}
Population genetics and quantitative genetics have mostly been developed in parallel throughout the twentieth century. Interpreting modern GWAS results requires understanding how the two interact, that is, how selection on a polygenic trait translates on the dynamics of allele frequencies \cite{Sella2019}. The historical approach to this problem, which we call the trait's eye-view, can be traced back to the works of Latter \cite{latter1960natural} and Bulmer \cite{Bulmer_1972} and has been applied to GWAS in \cite{Simons2018}. It consists in making assumptions on the distribution of the trait (typically a normality assumption), and conditional on the trait to model the evolution of the genes. Here we take the gene's-eye view \cite{geneeye}, meaning we start from a finite model of coupled equations representing the evolution of gene frequencies, and then let the number of loci go to infinity. In this setting, the marginal fitness of an allele depends not just on intrinsic properties of this allele but on the mean behavior of the other alleles within the population. Theorem \ref{thm:strongcvgrho} describes this limit with equation (\ref{eq:McKean-limitW}), which is a Wright-Fisher equation typical of population genetics. The behavior of the trait then emerges from the distribution of allele frequencies. This is seen in Theorem \ref{thm:geneticvariance} which describes the evolution of the trait once the behavior of the genes is known. From there we obtain equation (\ref{eq:Lande}) which is a typical equation of quantitative genetics. To the best of our knowledge, this gene-centric approach is new. We do note similarities with the ideas of the Dynamic Maximum Entropy (DME) approximation method \cite{bod2021dynamic}. This approximation from statistical physics consists in a joint modelling of deterministic macroscopic observables, corresponding in our setting to $\overline{s}(\mathscr{L}(f_t))$, and the evolution of a typical locus conditional on these observables. Concerning the description of stationarity, mean-field approximations to describe the stationary distribution of a quantitative trait under stabilizing selection were used in \cite{Nourmohammad_2013} (in a setting with symmetric mutations).

\subsubsection{Genetic structure}
Our approach raises the prospects of future exciting developments. The gene's-eye view in equation (\ref{eq:McKean-limitW}) can be extended without any difficulty to incorporate diploidy and dominance. It would also be important in terms of application to incorporate unequal allelic effects, replacing the set of genotypes $\{-1,+1\}^L$ with $\prod_{\ell\in [L]} \{-\alpha_\ell,+\alpha_\ell\}$ where the $\alpha_\ell$ have some distribution. This is not hard to do in our current setting, provided the allelic effects are bounded. However the ideal scaling would allow the distribution of allelic effects to have an exponential or even heavy (polynomial) tail. For instance, \cite{PARSONS2024117} found rough estimates of tail exponents in additive effects between 1 and 2.5, and \cite{Simons2018,Hayward2022} assume exponential tails. Similarly, we could account for mutation rate variation, letting $\theta^{(+)},\theta^{(-)}$ vary between the different loci. We will explore these extensions in future work.

\subsubsection{The strength of selection}
We chose in (\ref{eq:WLN}) to let the logfitness of an organism be of order $L$. This let us accommodate both stabilizing, directional and disruptive selection. In the limit equation (\ref{eq:Lande}), we see that for any initial condition, the mean trait value will evolve on the same timescale as $f_t$ in equation (\ref{eq:McKean-limitW}). In biological terms, the evolution of the trait is on the same timescale as genetic drift. This is in stark contrast with articles from the literature on stabilizing selection such as \cite{Hayward2022} which typically find that the evolution of the trait is much faster than the evolution of the underlying genes. There are differences with the underlying model, including in the way mutations are specified (they use the infinite-allele infinite-loci model and assume no mutational bias), but we believe that the results of \cite{Hayward2022} can be recovered from our model assuming the strength of selection to be of order $L^2$. Indeed, in Lemma \ref{lem:pibarom}, we compute the contribution of selection to the dynamics of locus $\ell$ under linkage equilibrium. Under stabilizing selection ($\kappa>0$ in (\ref{eq:U-square})), our result reads
$$
4\kappa\left(\eta-\frac{1}{L} \sum_{\ell'\in[L]}(2X_t^{\ell'}-1)\right) + \mathcal{O}\left(\frac{1}{L}\right)
$$
where we recall that $\mu_{\mathbf{X}_t}$ is the empirical distribution of $(p^\ell(\mathbf{X}_t))_{\ell\in[L]}$. But the same computations, replacing $L$ with $L^2$ in the base equation (\ref{eq:SDEfull}) yield the following selection coefficient at locus $\ell$
\begin{equation}\label{eq:selcoefL2}
4\kappa\left(L\eta- \sum_{\ell'\in[L]}(2X_t^{\ell'}-1)\right)  + 2\kappa \left(X_t^\ell-\frac{1}{2}\right)
\end{equation}
This corresponds to the selection coefficient as computed by \cite{Hayward2022}.
The second term is known to biologist as Robertson's underdominant term \cite{robertson1956effect}, and has been crucial for the evolutionary interpretation of GWAS results \cite{Simons2018}. 

Extending our results to the case where selection is of order $L^2$ requires handling the possibly degenerate first term of (\ref{eq:selcoefL2}). Furthermore, our control on linkage disequilibrium from Section \ref{sec:controlL2itself} will be affected, meaning the requirements on the strength of recombination $\rho$ for propagation of chaos will be stronger than in Theorem \ref{thm:strongcvgrho}.

Returning to the case when selection is of order $L$, if we assume $(\theta^{(+)}-\theta^{(-)})/(\theta^{(+)}+\theta^{(-)})\neq z^*$ in (\ref{eq:U-square}) and $X$ has distribution given by the stationary solution in Corollary \ref{cor:statio}, then
$$
\EE[2X-1-z^*] = \mathcal{O}(1)
$$
whereas Theorem \ref{thm:geneticvariance} states
$$
\sigma^2 = \mathcal{O}(1/L).
$$
This means in particular that at equilibrium, the population is very far from the selection optimum $|\EE[Z(g)]-z^*|\gg \sigma$. This corresponds to the drift-barrier hypothesis from \cite{driftbarrier}. The idea is that when selection is very weak, it cannot overpower the forces of mutation and genetic drift, and therefore the population at equilibrium remains very far from the optimum. This hypothesis was developed specifically for the evolution of the mutation rate, meaning $Z(g)$ corresponds to the probability of new mutations, making the models required to study this much more involved than ours (see for instance \cite{DAWSON1998143}). 

\subsubsection{The suppression of linkage by recombination}
Obtaining equation (\ref{eq:McKean-limitW}) required strong recombination, effectively enforcing independent between most loci. This prevents the formation of "linkage blocks". We do not believe our result to be optimal. For instance, if all loci are evenly spaced, except two loci $\ell_1,\ell_2$ so that $r_{\{\ell_1,\ell_2\}}=0$, then $r^{**} =0$, even though we expect that these two loci should not matter in the grand scheme of things. Furthermore, equation (\ref{eq:McKean-limitW}) seems to fit the dynamics of simulations even when $\rho$ is of order $L$ (see Fig \ref{fig:propagation_chaos}). 

Finding the optimal scaling for $\rho$ will be a daunting task. The suppression of linkage blocks by recombination was masterfully described in \cite{evans2013mutation}, in a very broad deterministic setting with recombination, bounded selection (whereas ours is of order $L$) and point mutations (no genetic drift). This approach assumes very rare mutations: a newborn's mutations are given by a point process on $[0,1]$ (representing the positions of the mutations along the chromosome). The main result is that linkage will be negligible if, loosely speaking, recombination separates two new mutations before a third one occurs. Statistical physicists have studied polygenic adaptation using Random Energy Models. These models typically assume the fitness of a genotype $g$ to be of the form
$$
W(g)=\sum_{\ell_1,\ell_2} f_{\ell_1\ell_2}g^{\ell_1}g^{\ell_2}
$$
for i.i.d random coefficients $(f_{\ell_1\ell_2})_{\ell_1,\ell_2\in [L]}$.
At least two phase transitions were identified for low recombination, which they call the transition from quasi-linkage equilibrium to clonal condensation or to non-random coexistence \cite{Dichio_2023}. These transitions see the appearance of very fit combination of genes which disproportionally contribute to the population without recombination breaking them up fast enough. The distinction between the two is that clonal condensation sees the appearance of true clones whereas in non-random coexistence a cloud of fit genotypes dominates the population. The phase transition occurs when selection is of the same order as recombination, corresponding in our system to $\rho\sim L$.

\subsection{Outline of the paper}
Our paper is organized as follows. In Section \ref{sec:Ltoinf}, we formally introduce McKean-Vlasov diffusions which generalize (\ref{eq:McKean-limitW}) and prove well-posedness of the associated martingale problem. In Section \ref{sec:StrongRec}, we prove Theorems \ref{thm:strongcvgrho} and \ref{thm:geneticvariance}.

\section*{On notation}
\subsection*{Summary of notations}

\begin{tabular}{l|l}
     $\square_A$& For $A\subseteq [L]$, the set $\{-1,+1\}^A$ \\\hline
     $\gamma|_A$& For $A\subseteq [L]$ and $\gamma\in \square_{[L]}$, the restriction of $\gamma$ to $A$ \\\hline
     $\XX^A$& The set of probability measures on $\square_A$, denoted $\mathbf{x}=(x(\gamma))_{\gamma\in \square_A}$ \\\hline
     $\mathbf{x}^A = (x^{A}(\gamma))_{\gamma\in \square_A}$& For $\mathbf{x}\in \RR^{\square_[L]}$, the marginal of $\mathbf{x}$ on $\square_A$ \\\hline
     $p^\ell(\mathbf{x})$& The frequency of the $+1$ allele at locus $\ell$: it is the same as $x^{\{\ell\}}(+1)$ \\\hline
     $\pi(\mathbf{x})$& The linkage equilibrium projection of $\mathbf{x}$ (see Section \ref{sec:linkeq}) \\\hline
    $\mathbf{x}[\varphi (g)]$ & For $\mathbf{x}\in \XX^{[L]}$, $\varphi$ a function on $\square_{[L]}$, this is the expectation of $\varphi(g)$ \\& where $g$ has law $\mathbf{x}$.     \\\hline
    $\mathbf{Var}_{\mathbf{x}},\mathbf{Cov}_{\mathbf{x}}$ & The variance and covariances associated with $\mathbf{x}$ (applied to functions of $g$) \\\hline
    $\mu_{\mathbf{x}}$& The allelic law:\\& $\mu_{\mathbf{x}}:= \frac{1}{L}\sum\limits_{\ell\in [L]} \delta_{p^\ell(\mathbf{x})}$
    \\\hline
     $\mathscr{L}(X)$ & The law of a variable $X$.\\\hline
     $\mathcal{O}$ & We write $\varphi_1=\mathcal{O}(\varphi_2)$ iff
     there is a constant $C>0$\\&\;
     such that $|\varphi_1|\leq C|\varphi_2|$ as some parameters of $\varphi_1$ and $\varphi_2$ are scaled.\\\hline
    $o$ & We write $\varphi_1=o(\varphi_2)$ iff there is a function $h$ such that $|\varphi_1|\leq h|\varphi_2|$ and $h\to 0$.\\\hline
    $C$& A positive constant \textbf{whose value may change from one line to the next}.\\\hline
\end{tabular}\vspace{10mm}

\subsection*{Marginals and restrictions}
For a genotype $\gamma\in\square_{[L]}$ and a subset $A\subseteq [L]$, $\gamma|_A:=(\gamma^\ell)_{\ell\in A}\in\square_A$ is the restriction of $\gamma$ to $A$.

A population $\mathbf{x}\in\XX^{[L]}$ can be seen as a vector of $\RR^{\square_{[L]}}$, written $(x(\gamma))_{\gamma\in\square_{[L]}}$. It can also be seen as a probability on $\square_{[L]}$. We write $\mathbf{x}[F(g)]$ for the expectation of $F(g)$, where $g$ is a random variable on $\square_{[L]}$ with law $\mathbf{x}$ and $F$ is a function on $\square_{[L]}$. We similarly define $\mathbf{Cov}_{\mathbf{x}}$ and $\mathbf{Var}_{\mathbf{x}}$ to be the covariance and variance associated with $\mathbf{x}$, evaluated on functionals of $g$.

For $\mathcal{I}\subseteq [L]$, $\mathbf{x}\in\RR^{\square_{[L]}}$, we let $\mathbf{x}^{\mathcal{I}} \equiv (x^{\mathcal{I}}(\gamma))_{\gamma\in\square_{\mathcal{I}}}$ be the marginal of $\mathbf{x}$ on $\mathcal{I}$. For $\mathbf{y}\in\RR^{\square_{[L]}}$ we may then define $\mathbf{x}^{\mathcal{I}}\otimes \mathbf{y}^{\mathcal{I}^c}$ as the product vector, that is, the vector such that its $\gamma-$th coordinate is
$$
\sum_{\substack{\gamma_1,\gamma_2\in\square_{[L]} \\
\gamma_1|_{\mathcal{I}}=\gamma|_{\mathcal{I}}\\
\gamma_1|_{\mathcal{I}}=\gamma|_{\mathcal{I}^c}}} x(\gamma_1)y(\gamma_2)
$$

\textbf{The parameters}\\
\begin{tabular}{c|l}
     $L$& The number of loci and the strength of selection\\
    & \textbf{Recombinator} $R$ \textbf{parameters}\\
     $\rho$& The strength of recombination (it depends on $L$, see Theorem \ref{thm:strongcvgrho})\\
     $\nu$& The measure associated to recombination (it depends on $L$)\\
     &\textbf{Mutator} $\Theta$ \textbf{parameters}\\
     $\theta^{(+)},\theta^{(-)}$& The mutation rate from $-1$ to $+1$ and back. \\
     &\textbf{Selector} $S$ \textbf{parameters}\\
     $U$& Under quadratic selection (\ref{eq:defquadraW}), $W(\gamma)=U(Z(\gamma))$ where $U$ is a polynomial of order 1 or 2\\
\end{tabular}

\section{Large genome limit \label{sec:Ltoinf}}
In this section, we describe McKean-Vlasov diffusions and the associated non-linear IPDEs, making sense of equation (\ref{eq:McKean-limitW}) which describes the limit behavior of our system as $L\to+\infty$. We fix two bounded measurable functions $a:\RR\to \RR_+$ and $b:\RR\times \mathcal{P}(\RR)\to \RR$ and a compactly supported probability measure $m_0\in\mathcal{P}(\RR)$.

\subsection{McKean-Vlasov diffusions \label{sec:mckean}}
Our aim is to give a sense to the following McKean-Vlasov SDE
\begin{eqnarray}\label{eq:SDEfull}
\dd f_t =& b(f_t,\xi_t)\dd t + \sqrt{a(f_t)}\dd B_t\qquad;\qquad \mathscr{L}(f_0)=m_0    \nonumber \\
& \mbox{with}  \ \xi_t = \mathscr{L}(f_t) \mbox{, the law of }f_t
\end{eqnarray}

\bigskip

Before going into (\ref{eq:SDEfull}), we first make a small detour and  consider an alternative problem. Let us now consider the following SDE
\begin{equation}\label{eq:toySDE}
\dd\hat{f}_t = b(\hat{f}_t,\zeta_t)\dd t + \sqrt{a(\hat{f}_t)}\dd B_t    \qquad;\qquad \mathscr{L}(\hat{f}_0)=m_0
\end{equation}
with $\zeta \in \DD([0,T],\mathcal{P}( \RR))$.
Recall that (\ref{eq:toySDE}) admits a weak solution on $[0,T]$ iff there exists a filtration $(\mathscr{F}_t;t\geq0)$ and an adapted  pair $(\hat{f},B)$,  
such that $B$ is a Brownian motion and (\ref{eq:toySDE}) is satisfied.
Theorem 6.1.6 of \cite{stroock1997multidimensional} implies the existence of weak solutions to (\ref{eq:toySDE}) for any initial condition $m_0$ as long as $a$ and $b(\cdot,m)$ are continuous for any $m\in \mathcal{P}(\RR)$. 

Given $\zeta\in \mathcal{P}(\RR)$, let us now define the following differential operator
$$\forall \varphi\in \mathcal{C}^2_c(\RR),\quad G_{\zeta} \varphi (x) := b(x,\zeta)\varphi'(x) + \frac{1}{2}a(x) \varphi''(x).$$
where $ \mathcal{C}_c^2(\RR)$ is the space of $\RR-$valued, compactly supported, continuously twice differentiable functions on $\RR$.

The existence and uniqueness in law of a solution to (\ref{eq:toySDE}) can be investigated through the associated martingale problem. For $\zeta\in\DD([0,T],\mathcal{P}(\RR))$, we say $\hat{f}$ solves the martingale problem for $(G_{\zeta_t})_{t\in[0,T]}$ iff
\begin{center}
    $(*) \quad \forall \varphi\in\mathcal{C}^2_c(\RR),\quad \varphi(\hat{f}_t)-\varphi(\hat{f}_0)-\int_0^t\dd u\,G_{\zeta_u}\varphi(\hat{f}_u)$ is a martingale in the filtration of $\hat{f}$ and $\mathscr{L}(f_0)=m_0$
\end{center}
Existence and uniqueness in law of a weak solution to (\ref{eq:toySDE}) is equivalent to existence and uniqueness of a solution to $(*)$ (see Theorem 4.5.2 in \cite{stroock1997multidimensional}).

\bigskip

Analogously, one can define a weak solution 
to the McKean-Vlasov SDE (\ref{eq:SDEfull}).
 Existence and uniqueness of weak solutions to the SDE (\ref{eq:SDEfull}) again boil down to the existence and uniqueness of solutions to the mean-field Martingale problem
\begin{center}
     $(**)\quad \forall \varphi\in\mathcal{C}^2_c(\RR),\quad \varphi(f_t)-\varphi(f_0)-\int_0^t\dd u\,G_{\xi_u}\varphi(f_u)$ is a martingale in the filtration of $f$\\
    $\ \ \ \ \ \ \ $ with $\mathscr{L}(f_0)=m_0$ and $\mathscr{L}(f_t) = \xi_t$\ \ \ \ \ \ 
\end{center}

\subsection{Weak solutions to non-linear IPDEs \label{sec:weaksolutionIPDE}}
We say a measure-valued process $\xi\in\DD([0,T],\mathcal{P}(\RR))$ is a {\it weak} solution to the non-linear IPDE
\begin{equation}\label{eq:PDEfull}
\partial_t u_t(x) = -\partial_x (b(x,u_t(\cdot))u_t(x)) + \frac{1}{2}\partial_{xx}(a(x)u_t(x))
\qquad;\qquad u_0=m_0
\end{equation}
if $t\mapsto \xi_t$ is continuous, $\xi_0=m_0$ and we have
$$
\forall \varphi\in \mathcal{C}^2_c(\RR),\ \qquad
\frac{\dd}{\dd t}<\xi_t,\varphi> \quad =\quad  <\xi_t,b(\cdot,\xi_t)\varphi'> 
+ \frac{1}{2} <\xi_t,a \varphi''>
$$

Consider $(f_t)_{t\in[0,T]}$ a weak solution to the McKean-Vlasov equation (\ref{eq:SDEfull}).
Then $(\mathscr{L}(f_t))_{t\in[0,T]}$ is a weak solution to the non-linear Fokker-Planck IPDE (\ref{eq:PDEfull}).

Let us now discuss the reciprocal statement: whether a weak solution to (\ref{eq:PDEfull}) is necessarily associated with a weak solution to (\ref{eq:SDEfull}).
We will need
\begin{thm}[Superposition principle (see Theorem 2.5 of \cite{superposition_principle})]
Consider $\tilde{a},\tilde{b}$ measurable bounded functions on $\RR\times [0,T]$ and define the generator
$$G_t \varphi(x) := \tilde{b}(x,t)\varphi'(x) + \frac{1}{2}\tilde{a}(x,t)\varphi''(x)$$
Let $\xi\in\DD([0,T],\mathcal{P}(\RR))$ be a weak solution to
$$
\partial_t u_t(x) = -\partial_x(\tilde{b}(x,t)u_t(x)) + \frac{1}{2}\partial_{xx}(\tilde{a}(x,t)u_t(x))
 \qquad;\qquad u_0 = m_0
$$
Then there exists a solution to the martingale problem
\begin{center}
     $\forall \varphi\in\mathcal{C}^2_c(\RR),\quad \varphi(f_t)-\varphi(f_0)-\int_0^t\dd u\,G_{u}\varphi(f_u)$ is a martingale in the filtration of $f$.
\end{center}
such that $\mathscr{L}(f_t) = \xi_t$.
\end{thm}
\begin{rem}
Our notion of weak solution corresponds to the notion of narrowly-continuous weak solution in \cite{superposition_principle}. 
\end{rem}
Let $\xi_t$ be a weak solution to (\ref{eq:PDEfull}). Then in particular $\xi$ is a weak solution to the linear PDE
$$
\partial_t u_t(x) = -\partial_x (b(x,\xi_t)u_t(x)) + \frac{1}{2}\partial_{xx}(a(x)u_t(x))
\qquad;\qquad u_0=m_0
$$
We may therefore apply the superposition principle to $\xi$ to find a solution $(f_t)_{t\in[0,T]}$ to the martingale problem 
\begin{center}
     $\forall \varphi\in\mathcal{C}^2_c(\RR),\quad \varphi(f_t)-\varphi(f_0)-\int_0^t\dd u\,G_{\xi_u}\varphi(f_u)$ is a martingale in the filtration of $f$
\end{center}
such that $f_t$ has law $\xi_t$. It follows that $f_t$ is a solution to $(**)$.

This shows that if we prove uniqueness of the solutions to $(**)$, we will automatically get uniqueness of solutions to (\ref{eq:PDEfull}).
 
\subsection{Uniqueness of solutions to Wright-Fisher-type McKean-Vlasov SDE }
There is a large literature on McKean-Vlasov diffusions, but Wright-Fisher-type SDEs such as ours require special handling because the Brownian coefficient is degenerate and not Lipschitz. This motivates the next theorem.

We metrize the weak topology on $\mathcal{P}(\RR)$ with the total variation distance (see \cite{chaintron2021propagation1}, proposition 4)
$$\forall \xi,\zeta\in \mathcal{P}(\RR), \ \quad 
\mathcal{D}_1(\xi,\zeta):=\sup_{||\varphi||_\infty \leq 1} |<\xi-\zeta,\varphi>|
$$
with $\left<\xi,\varphi\right>=\int_0^1 \varphi(x)\xi(\dd x)$, where the supremum is over all measurable functions $\varphi$ bounded by 1.

\begin{prop}\label{prop:uniqueness}
    Consider three bounded measurable function $a,c,\tilde{\Theta}:\RR\to \RR_+$ and a Lipschitz function  $\tilde{s}:\mathcal{P}(\RR)\to \RR$, and define $b(x,\zeta_t):=\sqrt{a(x)}c(x)\tilde{s}(\zeta_t) + \tilde{\Theta}(x)$. Assume that for any fixed $\zeta\in \DD([0,T],\mathcal{P}(\RR))$, the linear martingale problem (*) admits a unique weak solution.
    Then the McKean-Vlasov SDE (\ref{eq:SDEfull}) admits at most one weak solution and the non-linear PDE (\ref{eq:PDEfull}) admits at most one weak solution.
\end{prop}
\begin{proof}
We prove uniqueness of solutions by making a Girsanov transform, followed by Gr\"onwall's Lemma. This sort of proof was already used in \cite{girsanovuniqueness}.
 
Let us consider two solutions to the Martingale problem $(**)$ with initial distribution $m_0$ and denote by $\xi$ and $\zeta$ their respective laws.
We can construct two weak solutions $(f^\xi, B^{\xi})$ and $(f^\zeta, B^{\zeta})$ to the McKean-Vlasov SDE (\ref{eq:SDEfull}). 
 Consider a test function $\varphi$. We wish to bound the distance
$$ | <\xi_t-\zeta_t,\varphi> | = \left|\EE\left[\varphi(f^\xi_t)\right] - \EE\left[\varphi(f^\zeta_t)\right]\right|$$
and show that it must be $0$.

\textbf{Step 1: Girsanov transform.}
Consider the classical Cameron-Martin-Girsanov change of measure
\begin{align*}
    M_t:=&\int_0^t \;\frac{b(f^\zeta_u,\xi_u)-b(f^\zeta_u,\zeta_u)}{\sqrt{a(f^\zeta_u)}}\dd B_u^\zeta \\
        \frac{\dd\QQ}{\dd\PP}\quad :=&\quad \exp[M_T-\frac{1}{2}\braket{M}_T^{QV}]
\end{align*}
where $\braket{\cdot}^{QV}$ is the quadratic variation. Since for any $x\in \RR$, $\frac{b(x,\xi_u)-b(x,\zeta_u)}{\sqrt{a(x)}}= c(x)\left(\tilde{s}(\xi_u)-\tilde{s}(\zeta_u)\right)$ is bounded, $\QQ$ is well-defined since $M_t$ satisfies Novikov's condition.

Theorem 6.4.2 of \cite{stroock1997multidimensional}  implies that $f^\zeta$ under $\QQ$ satisfies the martingale problem
\begin{center}
$\forall \varphi\in\mathcal{C}^2_c(\RR),\quad \varphi(f_t)-\varphi(f_0)-\int_0^t\dd u\,G_{\xi_u}\varphi(f_u)$ is a martingale with $\mathscr{L}(f_0)=m_0$
\end{center}
Since we assumed that this linear martingale problem has a unique solution, it follows that $f^\zeta$ under $\QQ$ has the same law as $f^\xi$ under $\PP$. We thus obtain
$$\EE\left[\varphi(f^\xi_t)\right] = {\QQ}\left[\varphi(f^\zeta_t)\right] =
\EE\left[\varphi(f^\zeta_t)e^{M_t-\frac{1}{2}\braket{M}_t^{QV}}\right].$$

\textbf{Step 2: Gr\"onwall's Lemma.} We now alleviate notations by writing $f^\zeta\equiv f$.
From the previous discussion,
$$| <\xi_t-\zeta_t,\varphi> |=\left|\EE\left[\varphi(f_t)e^{M_t-\frac{1}{2}\braket{M}_t^{QV}}\right] - \EE\left[\varphi(f_t)\right]\right|$$
Write $\mathcal{E}_t:= e^{M_t-\frac{1}{2}\braket{M}_t^{QV}}$. We have
\begin{align*}
    |<\xi_t-\zeta_t,\varphi>| =& \left|\EE\left[\varphi(f_t) \left(\mathcal{E}_t - 1\right)\right]\right|\\
    \leq& ||\varphi||_\infty\;\EE\left[\left|e^{M_t-\frac{1}{2}\braket{M}_t^{QV}}-1\right|\right]\\
    =& ||\varphi||_\infty\; \EE\left[ \left|\int_0^t \dd M_u\; \mathcal{E}_u\right|\right]\\
    =& ||\varphi||_\infty\; \EE\left[ \left|\int_0^t \dd B_u\; c(f_u)(\tilde{s}(\xi_u)-\tilde{s}(\zeta_u))\mathcal{E}_u\right|\right]\\
    \leq& ||\varphi||_\infty\; \EE\left[ \left(\int_0^t \dd B_u\; c(f_u)(\tilde{s}(\xi_u)-\tilde{s}(\zeta_u))\mathcal{E}_u\right)^2\right]^{\frac{1}{2}}\\
    \leq& ||\varphi||_\infty \left(\EE\left[ \int_0^t \dd u\; c(f_u)^2\left(\tilde{s}(\xi_u)-\tilde{s}(\zeta_u)\right)^2\mathcal{E}_u^2\right]\right)^{\frac{1}{2}}\\
    \leq& C_{\tilde{s}}||\varphi||_\infty\;||c||_\infty\left(\int_0^t \dd u \; \mathcal{D}_1(\xi_u,\zeta_u)^2\EE\left[\mathcal{E}_u^2\right]\right)^{\frac{1}{2}}
\end{align*}
where $C_{\tilde{s}}$ is the Lipschitz constant of $\tilde{s}$. We used in the third line that $\mathcal{E}_t$ is the exponential martingale associated with $M_t$, and in the fifth line we used the Cauchy-Schwarz inequality. Since $\EE\left[\mathcal{E}_u^2\right]$ has uniform bounds on $[0,T]$,
we thus obtain
$$<\xi_t-\zeta_t,\varphi>^2\;\leq ||\varphi||_\infty^2C^2\int_0^t \dd u\; \mathcal{D}_1(\xi_u,\zeta_u)^2$$
for some constant $C>0$. Taking the supremum over $||\varphi||_\infty\leq 1$ yields
$$\mathcal{D}_1(\xi_t,\zeta_t)^2\;\leq C\int_0^t \dd u\; \mathcal{D}_1(\xi_u,\zeta_u)^2$$

The result follows from Gr\"onwall's lemma.
\end{proof}

We obtain as a Corollary the well-posedness of (\ref{eq:McKean-limitW}).
\begin{cor}\label{cor:uniqueness}
For $m_0\in \mathcal{P}([0,1])$, there exists a unique weak solution  to the McKean Vlasov problem (\ref{eq:McKean-limitW})
    \begin{gather*}
\dd f_t \ = \ \overline{s}(\mathscr{L}(f_t))  \  f_t(1-f_t) \dd t +
    \overline{\Theta}(f_t)\dd t
    +\sqrt{f_t(1-f_t)}\;\dd B_t \qquad;\qquad \mathscr{L}(f_0) = m_0
\end{gather*}
In particular $\mathscr{L}((f_t)_{t\in [0,T]})$ is a weak solution to the IPDE (\ref{eq:McKean-limitPDE}) on $[0,1]$
\begin{gather*}
\partial_t u_t(x) = -\partial_x\left[\left(\overline{s}(u_t(\cdot))x(1-x) + \overline{\Theta}(x)\right)u_t(x)\right]
+ \frac{1}{2}\partial_{xx} \left(x(1-x)u_t(x)\right)
\end{gather*}
with initial condition $m_0$.
\end{cor}
\begin{proof}
    Existence will be obtained from a convergence argument in the next section (see Theorem \ref{thm:strongcvgrho}).
    
    Let $a(x)=c(x)^2=x(1-x)\mathbf{1}_{[0,1]}(x)$ and $\tilde{\Theta}=\overline{\Theta},\tilde{s}=\overline{s}$ in Proposition \ref{prop:uniqueness}. Note that for $\xi\in\mathcal{P}([0,1])$, $\xi\mapsto <\xi,2\mbox{Id}-1>$ is Lipschitz. Finally, note that for any fixed $\zeta\in\DD([0,T],\mathcal{P}([0,1]))$, the equation 
$$
\dd f_t \ = \ \overline{s}(\zeta_t)  \  f_t(1-f_t) \dd t +
    \overline{\Theta}(f_t)\dd t
    +\sqrt{f_t(1-f_t)}\;\dd B_t \qquad;\qquad \mathscr{L}(f_0) = m_0
$$    
is a Wright-Fisher diffusion with a unique weak solution.
    
    We may therefore apply Proposition \ref{prop:uniqueness}. 
    The IPDE (\ref{eq:McKean-limitPDE}) is the Fokker-Planck equation associated with (\ref{eq:McKean-limitW}). The uniqueness of its solution is given by the discussion in Section \ref{sec:weaksolutionIPDE}.
\end{proof}

\subsection{Stationary distribution for symmetric quadratic selection \label{sec:statio}}
We focus on symmetric quadratic selection.
\begin{cor}\label{cor:statio}
Assume $\theta^{(+)},\theta^{(-)}>0$ and 
    consider the case of symmetric quadratic selection
    $$\overline{s}(\xi)=-2\kappa<\xi,2\mbox{Id}-1>$$    
for some parameter $\kappa\in\RR$.
Then
\begin{itemize}
    \item (stabilizing selection) Suppose $\kappa \geq 0$. Then $\chi$ has a unique fixed point at 0.
    \item (disruptive selection) Assume $\theta^{(+)}=\theta^{(-)}$. Let $\kappa_c := -\frac{4\theta^{(+)}+1}{2}$. Then we can find $\delta>0$ such that if $\kappa\in(\kappa_c-\delta,\kappa_c]$, $\chi$ has at least three fixed points.
\end{itemize}    
\end{cor}
\begin{proof}
Recall the definition of $\Pi_y$ in (\ref{def:Piy}) and the definition of $\chi$ in (\ref{eq:tmpchi}).
Let  $F$ be the cumulant generating function of $\Pi_0$
$$
F(y) := \ln(<\Pi_0, \exp(y\;\mbox{Id})>)
$$
A quick computation shows that $\chi(y)=-2\kappa(2F'(2y)-1)$. In particular
\begin{equation}\label{eq:chi'}
\chi'(y) = -8\kappa F''(2y)
\end{equation}
It is also easy to see that $F''(2y)$ is the variance of $\Pi_y$, and in particular is positive.
In the $\kappa\geq 0$ case, we get that $\chi$ is non-increasing. In particular it will have a single fixed point. 

\medskip

We turn to the case $\theta^{(+)}=\theta^{(-)}$. First, notice that by symmetry we will always have $\overline{s}(\Pi_0) = 0$. 
In particular, for any $\kappa$, 0 is a fixed point of $\chi$.
Further,
\begin{equation}\label{eq:taylorchi}
\chi(y) = y\chi'(0) 
+ \frac{y^2}{2}\chi''(0)
+ \frac{y^3}{6}\chi'''(0)
+ o(y^3)
\end{equation}
The variance of $\Pi_0$ is $F''(0)=\frac{1}{4(4\theta^{(+)}+1)}$.  From (\ref{eq:chi'}) we find
$$
\chi'(0) = \frac{\kappa}{\kappa_c}
$$
We similarly compute 
$$
\chi''(0) = -16\kappa\; F'''(0) \qquad;\qquad
\chi'''(0) = -32\kappa F''''(0)
$$
Since $F'''(0)$ is the skew of $\Pi_0$, it is 0 by symmetry. Recall $F''''(0)$ is the fourth cumulant of $\Pi_0$. This can be seen to be negative for symmetric Beta distributions. We then rewrite (\ref{eq:taylorchi}) as
$$
\chi(y) - y = y\frac{\kappa-\kappa_c}{\kappa_c}
- \frac{y^3}{6}\kappa32F''''(0)
+ o(y^3)
$$
When $\kappa\leq \kappa_c$, the two terms on the right-hand side are of opposite signs. The result follows.
\end{proof}

\section{Convergence to the McKean-Vlasov SDE under strong recombination \label{sec:StrongRec}}
In this section we prove Theorem \ref{thm:strongcvgrho} (in Section \ref{sec:cclstrongrec}), and  Theorem \ref{thm:geneticvariance} (in Section \ref{sec:phenotype}). We start with an outline of the main step of the proofs.

\subsection{Heuristics and outline of the proof} \label{sec:linkeq}
Recall that we consider a multidimensional SDE of the form
$$
\dd {\bf X}_t \ = \ (\rho R({\bf X}_t) + \Theta({\bf X}_t) + L S({\bf X}_t))\dd t + \Sigma(\mathbf{x})\dd\mathbf{B}_t
$$

\bigskip

{\bf Step 1.} For every $\mathbf{x}\in \XX^{[L]}$, define $\pi(\mathbf{x})$ as
$$
    \pi(\mathbf{x}) := \bigotimes_{\ell\in [L]} \mathbf{x}^{\{\ell\}}
$$
    i.e., $\pi(\mathbf{x})$ is the product measure whose one dimensional marginals coincides with the ones of $\mathbf{x}$.
    We extend the definition of $\pi$ to $\bigcup_{A\subseteq [L]} \XX^A$, such that
    \begin{equation}\label{eq:defpi}
    \forall \mathbf{x}\in \XX^A, \ \ \  \pi(\mathbf{x}) := \bigotimes\limits_{\ell\in A} \mathbf{x}^{\{\ell\}}.
    \end{equation}
$\pi$ gives the attractors of the recombinator $R$ as the following Lemma shows. 
\begin{lem}\label{lem:defnolink}
    For $\mathbf{x}_0\in \XX^{[L]}$, define $\mathbf{x}_t$ to be the unique solution to
\begin{align*}
\frac{\dd\mathbf{x}_t}{\dd t} =& R(\mathbf{x}_t)
\end{align*}
with initial condition $\mathbf{x}_0$. 
Then, provided the recombination measure $\nu$ is non-degenerate, $\mathbf{x}_t$ converges to  $\pi(\mathbf{x}_0)$ as $t\to +\infty$.
\end{lem}
\begin{proof}
We refer the reader to \cite{entropy_production} and the references therein.
\end{proof}
  
Under the strong recombination assumption, the driving force is recombination. From the previous result,  we may expect that 
$$
{\bf X}_t \ \approx \ \pi({\bf X}_t)
$$ 
so that the SDE should asymptotically diffuse on the stable manifold for the recombinator
$$
\Gamma^{[L]} := \{\mathbf{x}\in \XX^{[L]}\;|\quad \mathbf{x}=\pi({\bf x})\}.
$$
Biologically speaking, we expect the system to be at linkage equilibrium (LE) due to the overwhelming effect of recombination.

\bigskip

{\bf Step 2.} 
For $\ell\in[L]$, define 
$$
S^\ell({\bf x}) := S^{\{\ell\}}({\bf x})(+1)
$$
where we recall that $S^{\{\ell\}}$ is the marginal of $S$ on $\{\ell\}$.
We show in Corollary \ref{cor:SDEp(X)} that
\begin{equation*}
 \forall \ell\in[L], \ \   \dd p^\ell(\mathbf{X}_t) = \left(\overline{\Theta}(p^\ell(\mathbf{X}_t))+LS^\ell(\mathbf{X}_t)\right)\dd t + \sqrt{p^\ell(\mathbf{X}_t)(1-p^\ell(\mathbf{X}_t))}\dd\hat{B}_t^\ell
\end{equation*}
where $\hat B^{l}$ is a Brownian motion.
It is also easy to see that if ${\bf x}$ belongs to the LE manifold (see Lemma  \ref{lem:pibarom})
\begin{equation}\label{eq:s-vs-S}
L S^\ell({\bf x}) \approx \bar s(\mu_{{\bf x}})p^\ell({\bf x}) (1-p^\ell({\bf x}))
\end{equation}
It should follow that
\begin{equation*}
 \forall \ell\in[L], \ \   \dd p^\ell(\mathbf{X}_t) \approx \left(\overline{\Theta}(p^\ell(\mathbf{X}_t))+ \bar s(\mu_{{\bf X}_t})p^\ell(\mathbf{X}_t)(1-p^\ell(\mathbf{X}_t)))\right)\dd t + \sqrt{p^\ell(\mathbf{X}_t)(1-p^\ell(\mathbf{X}_t))}\dd\hat{B}_t^\ell \label{eq:one-dimensional-cor}
\end{equation*}
 
In order to derive a mean field approximation, it remains to prove 
that loci become decorrelated at the limit.
Define the linkage disequilibrium between $\ell_1\neq\ell_2\in[L]$ as
\begin{equation}
\label{def:D}
D^{\ell_1,\ell_2}({\bf x}) := \mathbf{Cov}_{\mathbf{x}}\left[\mathbbm{1}_{[g^{\ell_1}=+1]},\mathbbm{1}_{[g^{\ell_2}=+1]}\right] 
\end{equation}
where we recall that $\mathbf{Cov}_{\mathbf{x}}$ is the covariance of functionals of a random variable $g$ with law $\mathbf{x}$. We show (again in Corollary \ref{cor:SDEp(X)}) that
\begin{equation}
    \label{eq:D-LD}
\dd \left<p^{\ell_1}(\mathbf{X}),  p^{\ell_2}(\mathbf{X}) \right>_t^{QV} \ = \ D^{\ell_1,\ell_2}({\bf X}_t)\dd t = 0
\end{equation}
where the last equality holds provided that ${\bf X}_t$ is on the LE manifold $\Gamma^{[L]}$.
Putting everything together, if recombination is strong enough, we should expect a propagation of chaos principle to hold.

\bigskip

{\bf Technical ingredients.} The previous heuristics rely on the underlying assumption that $\mathbf{X}_t$ is on the LE manifold. Making this rigorous will raise one major  difficulty. We need to derive the conditions on the recombinator so that ${\bf X}_t$ remain close enough to the boundary so that the previous estimates remain valid.

The first step is the following proposition
which allows 
to justify (\ref{eq:s-vs-S}) and (\ref{eq:D-LD}) by controlling
\begin{equation}\label{eq:linkage-3}
Y^{A}_{t}:=||\mathbf{X}_t^{A}-\pi(\mathbf{X}_t^{A})||_2
\end{equation}
on every subset $A\subset [L]$ of size $2$ or $3$.

\begin{prop}\label{prop:boundDSwithL2}
We have
\begin{align}\label{eq:boundDSwithL2v1}
\forall \ell_1\neq \ell_2\in [L],\qquad\quad&
|D^{\ell_1,\ell_2}(\mathbf{x})|
\leq\quad ||\mathbf{x}^{\{\ell_1,\ell_2\}} - \pi(\mathbf{x}^{\{\ell_1,\ell_2\}})||_2\\
\forall \ell_0\in [L],\qquad\qquad&
\big|S^{\ell_0}(\mathbf{x})-S^{\ell_0}(\pi(\mathbf{x}))\big|\leq\quad
C\sum\limits_{\substack{A\subseteq [L]\smallsetminus \{\ell_0\}\\ 1\leq \#A\leq 2}}\frac{1}{L^{\#A}}
||\mathbf{x}^{\{\ell_0\}\cup A}-\pi(\mathbf{x}^{\{\ell_0\}\cup A})||_2 \label{eq:boundDSwithL2v2}
\end{align}
for some constant $C>0$ independent of $L$.
\end{prop}

To get control on (\ref{eq:linkage-3}), we will use the linearized recombinator $\nabla R$. The eigenvalues of $\nabla R$ were computed in \cite{Lyubich_1971} p. 107. We give a full spectral characterization of $\nabla R$ in Section \ref{sec:eigenrec}. 
In particular, we obtain that  the system is always attracted towards the LE manifold at a rate at least
\begin{equation}\label{eq:defrA}
    r_A:= \min_{\substack{\ell_1,\ell_2\in A\\ \ell_1\neq\ell_2}} r_{\{\ell_1,\ell_2\}}
\end{equation}
In Section \ref{sec:controlL2itself}, we use this estimate through a combination of It\^o's and Gr\"onwall's lemmas to get quantative bounds on  $(\ref{eq:linkage-3})$. In Section \ref{sec:cclstrongrec} we conclude the proof of Theorem \ref{thm:strongcvgrho} with standard martingale arguments to prove convergence of $(\mu_{\mathbf{X}_t})_{t\in[0,T]}$ and $(p^\ell(\mathbf{X}_t))_{t\in[0,T]}$. Finally,
we obtain Theorem \ref{thm:geneticvariance} in Section \ref{sec:phenotype}.

\begin{rem}
Katzenberger \cite{katzenberger1990solutions}  considered 
a generic SDE with a strong drift attracting the dynamics on an invariant manifold and derived a slow-fast principle for the stochastic evolution on the manifold. This sort of proof was already used in a population genetics context, in the case $\rho \to +\infty,L=2$ in \cite{ethier1979limit}. Our system presents two additional complexities. The first difficulty is that the dimension of the problem explodes exponentially with $L$ just as the strength of recombination becomes large. The second difficulty is that we not only need  $\mathbf{x} \approx \pi({\bf x})$, we actually need the difference to be small, of order $1/L$, because the strength of selection is of order $L$. We therefore require quantitative bounds on the linkage disequilibrium on any small set $A\subseteq [L]$. 
\end{rem}

\subsection{Evolution of the marginals\label{sec:evmarg}}
In this section we will derive the SDE for $p^\ell(\mathbf{X}_t)$. We will in fact study $\mathbf{X}_t^A$ for any $A\subseteq [L]$, of which $\{\ell\}$ is a special case. The reason why we need to study $\mathbf{X}_t^A$ for a general $A$ is because we will need to control the divergence from LE of $\mathbf{X}^A$ for small sets $A$ of size at most $3$.

Recall the definition of  $R,\Theta, \Sigma$
as defined in in the introductory Section \ref{sec:defmodel}.
For any subset $A$, we define the same quantity 
$\hat R^A,\hat \Theta^A, \hat \Sigma^A$ but on  the hypercube $\square_{A}$. For instance,the operators $\hat R^A, \hat \Theta^A : \XX^{A} \to \RR^{\square_{A}}$  
read
\begin{align*}
\forall \mathbf{x}\in \XX^A,\qquad \hat{\Theta}^{A}(\mathbf{x})(\gamma) :=&
|\theta|\sum\limits_{\ell\in A} 
    \left(\mathbf{x}^{[L]\smallsetminus\{\ell\}}\otimes \mathcal{L}_\theta -\mathbf{x}\right)\\
 \forall \mathbf{x}\in \XX^A \qquad \hat{R}^A(\mathbf{x}) :=&
    \sum\limits_{\emptyset \subsetneq \mathcal{I}\subsetneq A}  \ \ \nu^A(\mathcal{I})\left(
    \mathbf{x}^{\mathcal{I}} \otimes \mathbf{x}^{A\smallsetminus \mathcal{I}} - \mathbf{x} \right)
\end{align*}
where we recall that $\nu^A$ is the marginal of $\nu$ on $A$, and we take an empty sum to be equal to zero. Similarly, $\hat \Sigma^A$ is a function
$$
\hat \Sigma^{A} :  \ \XX^{\square_A} \to \   \mathcal{M}\left(\square_A\times \square_A,\RR^{\square_A}\right)
$$

\bigskip

Define  $S^A : \XX^{L} \to \RR^{\square_{A}}$
as
$$
\forall x\in \XX^{L}, \qquad
S^{A}({\bf x}) \ = \ (S({\bf x}))^{A}  
$$
so that $S^{A}({\bf x})$ is the generalized marginal of $S({\bf x})$ on $A$. A direct computation shows that 
\begin{align}\label{eq:SAformula}
\forall \mathbf{x}\in \XX^{[L]},\qquad 
S^{A}(\mathbf{x})(\gamma) =&   x^{A}(\gamma) \;\left(\mathbf{x}\big[W(g)\;\big| \ g|_A=\gamma\big]
 -\mathbf{x}\big[W(g)\big]\right)\\
=& \mathbf{Cov}_{\mathbf{x}}[W(g),\mathbbm{1}_{[ \ g|_A=\gamma]}]\label{eq:SAcov}
\end{align}

\begin{rem}
Note that  $S^A({\bf x})$  is the marginal on $A$ of $S(\mathbf{x})$.
However, it is not so for $\hat{R}^A$. $\hat{R}^A$ is defined on $\XX^A$ (not on $\XX^{[L]}$). To stress out the distinction, we write $\hat{R}^A$ and not $R^A$. The same goes for $\hat{\Theta}^A,\hat{\Sigma}^A$. 
\end{rem}

\begin{prop}\label{prop:consistency}
For $A\subseteq [L]$, there is a Gaussian process $\hat{\mathbf{B}}^A = (\hat{B}^{A}(\gamma_1\gamma_2))_{\gamma_1\neq \gamma_2\in\square_A}$ such that
\begin{equation}\label{eq:consistency}
\forall \gamma\in \square_{A},\qquad
  \dd\mathbf{X}_t^{A} = \left(\rho \hat{R}^{A}(\mathbf{X}_t^A) + \hat{\Theta}^{A}(\mathbf{X}_t^A) + LS^{A}(\mathbf{X}_t) \right)\dd t + \hat{\Sigma}^{A}(\mathbf{X}_t^A)\dd\hat{\mathbf{B}}_t^A
\end{equation}
Furthermore, $\hat{B}^{A}(\gamma_1,\gamma_2) = -\hat{B}^{A}(\gamma_2,\gamma_1)$ and $\hat{B}^A(\gamma_1,\gamma_2),\hat{B}^A(\gamma_3,\gamma_4)$ are independent Brownian motions whenever $(\gamma_1,\gamma_2)\notin \{(\gamma_3,\gamma_4),(\gamma_4,\gamma_3)\}$.
\end{prop}

\begin{proof}
We have
\begin{align}
\dd X_t^{A}(\gamma) =& \sum\limits_{\hat{\gamma}\in\square_{[L]}} \mathbbm{1}_{[\hat{\gamma}|_A = \gamma]} \dd X_t({\hat{\gamma}})\nonumber\\
=& \left(R(\mathbf{X}_t) + \Theta(\mathbf{X}_t)\right)^A(\gamma)\dd t
+ L S^A(\mathbf{X}_t)(\gamma)\dd t 
+ \sum\limits_{\substack{\hat{\gamma}_1, \hat{\gamma}_2\in\square_{[L]}\\\hat{\gamma}_1\neq \hat{\gamma}_2}}\mathbbm{1}_{[\hat{\gamma}_1^A=\gamma]} \sqrt{X_t(\hat{\gamma}_1) X_t(\hat{\gamma}_1})\dd B_t(\hat{\gamma}_1,\hat{\gamma}_2)
\label{eq:tmpconsistency1}
\end{align}
We first calculate the marginal effect of recombination. We use the Proposition 6 of \cite{Baakebaake}, where the following consistency relation is shown
\begin{equation*}
\left(R(\mathbf{x})\right)^A = \hat{R}^A(\mathbf{x}^A)
\end{equation*}
Secondly, recall that $\Theta$ is the generator corresponding to mutation: each locus mutates independently of the others, from $-1$ to $+1$ (resp. $-1$ to $+1$) at rate $\theta^{(+)}$ (resp. $\theta^{(-)}$). We can therefore expect a consistency property, by which taking the marginal effect of $\Theta$ on the loci in $A$, each locus in $A$ mutates independently of the rest with rates $\theta^{(+)},\theta^{(-)}$, which translates into
$$
(\Theta(\mathbf{x}))^{A} = \hat{\Theta}^A(\mathbf{x}^A)$$
Formally, this can be proved as follows.
\begin{align*}
    (\Theta(\mathbf{x}))^A(\gamma)
=
    |\theta|\sum\limits_{\ell\in [L]} 
     \left(\left(\mathbf{x}^{[L]\smallsetminus\{\ell\}}\otimes \mathcal{L}_\theta\right)^A -\mathbf{x}^A\right)
\end{align*}
For any $\ell \notin A$, we have 
$$
(\mathbf{x}^{[L]\smallsetminus \{\ell\}} \otimes \mathcal{L}_\theta)^A = \mathbf{x}^A
$$
This means the sum on $\ell\in [L]$ can be restricted to $A$, which yields $\hat{\Theta}^A(\mathbf{x}^A)$.

Finally, we turn to the Brownian term. This term corresponds to the equation for a neutral Wright-Fisher diffusion with $2^L$ alleles. It is well-known that the multi-allele Wright-Fisher diffusion admits a consistency property, by which if we group alleles together into $2^{\# A}$ families, the frequencies of these families behave like a $2^{\# A}$-allele Wright-Fisher diffusion. For the unconvinced reader we give a sketch of the proof.

For $\gamma_1,\gamma_2 \in \square_A$, define
$$
\dd\hat{B}_t^A(\gamma_1,\gamma_2) := 
\sum\limits_{\substack{\hat{\gamma}_1,\hat{\gamma}_2\in \square_{[L]}\\
\hat{\gamma}_1^A = \gamma_1\\ \hat{\gamma}_2^A = \gamma_2
}}
\sqrt{\frac{X_t(\hat{\gamma}_1)X_t(\hat{\gamma}_2)}{X_t^A(\gamma_1)X_t^A(\gamma_2)}} \dd B_t(\hat{\gamma_1},\hat{\gamma}_2)
$$
Check that $\hat{B}^A$ is formally well-defined because
$$\mathbbm{1}_{[\hat{\gamma}_1^A = \gamma_1,\hat{\gamma}_2^A = \gamma_2]}\frac{X_t(\hat{\gamma}_1)X_t(\hat{\gamma}_2)}{X_t^A(\gamma_1)X_t^A(\gamma_2)} \leq 1.$$
Recall that for any $\hat{\gamma}_1,\hat{\gamma}_2,\hat{\gamma}_3,\hat{\gamma}_4\in \square_{[L]}$
$$\dd \braket{B(\hat{\gamma}_1,\hat{\gamma}_2),B(\hat{\gamma}_3,\hat{\gamma}_4)}_t^{QV} =
(\mathbbm{1}_{[\hat{\gamma}_1=\hat{\gamma}_3,\hat{\gamma}_2=\hat{\gamma}_4]}-\mathbbm{1}_{[\hat{\gamma}_1=\hat{\gamma}_4,\hat{\gamma}_2=\hat{\gamma}_3]})\dd t$$
From there, it is straightforward to obtain for $\gamma_1,\gamma_2,\gamma_3,\gamma_4\in\square_A$
$$
\dd \braket{\hat{B}^A(\gamma_1,\gamma_2),\hat{B}^A(\gamma_3,\gamma_4)} = 
\left(\mathbbm{1}_{[\gamma_1=\gamma_3,\gamma_2=\gamma_4]} -\mathbbm{1}_{[\gamma_1=\gamma_4,\gamma_2=\gamma_3]}\right)\dd t.
$$
To conclude, let us show $\hat{\Sigma}^A(\mathbf{X}_t^A)\dd \mathbf{B}_t^A$ is equal to the Brownian term from (\ref{eq:tmpconsistency1}). We write for fixed $\gamma_1\in \square_A$ 
$$\sum\limits_{\substack{\gamma_2\in \square_{[L]}\smallsetminus \{\gamma_1\}}} \sqrt{X_t^A(\gamma_1)X_t^A(\gamma_2)} \dd\hat{B}_t^A(\gamma_1,\gamma_2) = 
\sum\limits_{\substack{\hat{\gamma}_1, \hat{\gamma}_2\in \square_{[L]}\\
\hat{\gamma}_1^A = \gamma_1 \neq \hat{\gamma}_2^A
}} \sqrt{X_t(\hat{\gamma}_1)X_t(\hat{\gamma}_2)} \dd B_t(\hat{\gamma}_1,\hat{\gamma}_2)
$$
We can extend the sum to the cases where $[\hat{\gamma}_2^A = \gamma_1,\hat{\gamma}_2 \neq \hat{\gamma}_1]$ because the terms $(\hat{\gamma}_1,\hat{\gamma}_2)$ and $(\hat{\gamma}_2,\hat{\gamma}_1)$ cancel out. This yields the Brownian term from (\ref{eq:tmpconsistency1}).
\end{proof}
\begin{rem}
We may notice that if $W=0$, then $S^A=\mathbf{0}$ and the equation for $\mathbf{X}_t^A$ is autonomous.
\end{rem}

We can apply the previous Proposition to $A=\{\ell\}$ and obtain an important Corollary. We need a Lemma
\begin{lem}[Fleming-Viot property \cite{flemingviot}]\label{lem:covF1F2}
    For functions $F_1$ and $F_2$ on $\square_{[L]}$ we have
    $$
    \dd \braket{\mathbf{X}[F_1(g)],\mathbf{X}[F_2(g)]}_t = \mathbf{Cov}_{\mathbf{X}_t}[F_1(g),F_2(g)]\dd t
    $$
\end{lem}
\begin{proof}
We have
\begin{align*}
\dd \braket{\mathbf{X}[F_1(g)],\mathbf{X}[F_2(g)]}_t 
=& 
\sum\limits_{\gamma_1,\gamma_2\in\square_{[L]}}
F_1(\gamma_1)F_2(\gamma_2)
\left( X_t(\gamma_1)\mathbbm{1}_{[\gamma_1=\gamma_2]}
-
X_t(\gamma_1)X_t(\gamma_2) \right)\dd t\\
=& 
\sum\limits_{\gamma_1\in\square_{[L]}}
F_1(\gamma_1)F_2(\gamma_1) X(\gamma_1)_t \dd t
\quad -
\sum\limits_{\gamma_1,\gamma_2\in\square_{[L]}}
F_1(\gamma_1)F_2(\gamma_2)X(\gamma_1)X(\gamma_2) \dd t
\end{align*}
where in the first equality we used that $\dd \braket{X(\gamma_1),X(\gamma_2)}_t = \left(\mathbbm{1}_{[\gamma_1=\gamma_2]}X_t(\gamma_1) -X_t(\gamma_1)X_t(\gamma_2)\right)\dd t$.
The last line is $\mathbf{Cov}_{\mathbf{X}_t}[F_1(g),F_2(g)]$.
\end{proof}

\begin{cor}\label{cor:SDEp(X)}
\begin{equation}
 \forall \ell\in[L], \ \   \dd p^\ell(\mathbf{X}_t) = \left(\overline{\Theta}(p^\ell(\mathbf{X}_t))+LS^\ell(\mathbf{X}_t)\right)\dd t + \sqrt{p^\ell(\mathbf{X}_t)(1-p^\ell(\mathbf{X}_t))}\dd\hat{B}_t^\ell
 \label{eq:SDEp(X)}
\end{equation}
with $\overline{\Theta}$ from Theorem \ref{thm:strongcvgrho} and $(\hat{B}^\ell)_{\ell\in[L]}$ a $L$-dimensional  Brownian motion with
\begin{equation}
\forall i\neq j\in[L], \ \dd \braket{p^{\ell_i}(\mathbf{X}),p^{\ell_j}(\mathbf{X})}_t^{QV} = D^{\ell_1,\ell_2}(\mathbf{X}_t)\dd t
\label{eq:covell1ell2}    
\end{equation}
where we recall $D^{\ell_1,\ell_2}$ from (\ref{def:D}).
\end{cor}
\begin{proof}[Proof of Corollary \ref{cor:SDEp(X)}]
We apply Proposition \ref{prop:consistency} to $A=\{\ell\}$ and consider the $+1$ coordinate, recalling that $p^\ell(\mathbf{x})=\mathbf{x}^{\{\ell\}}(+1)$). We also use the fact that $\hat{R}^{\{\ell\}}=\mathbf{0}$, i.e., recombination does not alter allele frequencies.
To get (\ref{eq:covell1ell2}), apply Lemma \ref{lem:covF1F2} with $F_i(g)=\mathbbm{1}_{[g^{\ell_i}=+1]}$ for $i\in\{1,2\}$.
\end{proof}

Comparing equation (\ref{eq:SDEp(X)}) and the desired limit equation (\ref{eq:McKean-limitW}), we see that we need two things. On the one hand, for any $\ell_1\neq \ell_2$ we must obtain that $\braket{p^{\ell_1}(\mathbf{X}),p^{\ell_2}(\mathbf{X})}_t^{QV} \to 0$. This will be achieved by controlling $D^{\ell_1,\ell_2}(\mathbf{X}_t)$. On the other hand, we need that
$$|LS^\ell(\mathbf{X}_t) - p^\ell(\mathbf{X}_t)(1-p^\ell(\mathbf{X}_t)) \overline{s}(\mu_{\mathbf{X}_t}) | \longrightarrow 0$$
This is much more difficult to obtain. We will get it by showing $S^\ell(\mathbf{X}_t)\simeq S^\ell(\pi(\mathbf{X}_t))$ and using the following Lemma

\begin{lem}\label{lem:pibarom}
We have
$$
\forall \ell\in [L],\qquad
LS^{\ell}(\pi(\mathbf{x})) = p^\ell(\mathbf{x})(1-p^\ell(\mathbf{x}))\overline{s}(\mu_{\mathbf{x}})+\mathcal{O}\left(\frac{1}{L}\right)
$$
where $\mathcal{O}$ is uniform in $\mathbf{x}$.
\end{lem}
\begin{proof}
    Recall from the definition of $\pi$ that $p^\ell(\pi(\mathbf{x}))=p^\ell(\mathbf{x})$. Applying this to equation (\ref{eq:SAformula}) for $\ell\in [L]$ we get
    \begin{align*}
    LS^{\ell}(\pi(\mathbf{x})) =& Lp^\ell(\mathbf{x})\left(\pi(\mathbf{x})[U(Z(g))|g^\ell=+1]-\pi(\mathbf{x})[U(Z(g))]\right)\\
    =& Lp^\ell(\mathbf{x})
    \Bigg(\pi(\mathbf{x})\left[U\left(\frac{1}{L} 
    + \frac{1}{L}\sum\limits_{\hat{\ell}\in [L]\smallsetminus\{\ell\}} g^{\hat{\ell}}\right)\Big|g^\ell=+1\right]\\&
    \qquad\qquad\qquad\qquad
    -\pi(\mathbf{x})\left[U\left(\frac{g^\ell}{L} + \frac{1}{L}\sum\limits_{\hat{\ell}\in[L]\smallsetminus\{\ell\}} g^{\hat{\ell}}\right)\right]\Bigg)\\
=& p^\ell(\mathbf{x})\times 2(1-p^\ell(\mathbf{x}))
    \times \pi(\mathbf{x})\left[ U'\left(\frac{1}{L}\sum\limits_{\hat{\ell}\in[L]\smallsetminus\{\ell\}} g^{\hat{\ell}}\right)
    \right] + \mathcal{O}\left(\frac{1}{L}\right)
\end{align*}
where in the third equality we used that $g^{\hat{\ell}}$ and $g^\ell$ are independent under $\pi(\mathbf{x})$ and $\pi(\mathbf{x})[g^\ell]=2p^\ell(\mathbf{x})-1$.
To conclude, write
\begin{align*}
\pi(\mathbf{x})\left[\frac{1}{L}\sum\limits_{\hat{\ell}\in[L]\smallsetminus\{\ell\}} g^{\hat{\ell}}
    \right] 
    =&\;  <\mu_{\mathbf{x}},2\mbox{Id}-1> + \mathcal{O}\left(\frac{1}{L}\right)
\end{align*}
Since we assumed $U$ to be a quadratic polynomial, then $U'$ is  of degree one and we get the result.
\end{proof}

We conclude this section by proving Proposition \ref{prop:boundDSwithL2} which states that we may control both $D^{\ell_1,\ell_2}(\mathbf{X}_t)$ and $S^\ell(\mathbf{X}_t)$ with $||\mathbf{X}_t^A-\pi(\mathbf{X}_t^A)||_2$ for small sets $A\subseteq [L]$.
\begin{proof}[Proof of Proposition \ref{prop:boundDSwithL2}]
The first inequality is readily obtained with
\begin{align*}
\forall \ell_1\neq\ell_2,\qquad
\left|D^{\ell_1,\ell_2}(\mathbf{x})\right| 
=& \left|\mathbf{Cov}_{\mathbf{x}}\left[\mathbbm{1}_{[g^{\ell_1}=+1]},\mathbbm{1}_{[g^{\ell_2}=+1]}\right]\right|\\
=& \left|\mathbf{x}\left[\mathbbm{1}_{[g^{\ell_1}=+1,g^{\ell_2}=+1]}\right] - \pi(\mathbf{x})\left[\mathbbm{1}_{[g^{\ell_1}=+1,g^{\ell_2}=+1]}\right]\right|\\
\leq& \left|\left|\mathbf{x}^{\{\ell_1,\ell_2\}}-\pi^{\{\ell_1,\ell_2\}}(\mathbf{x})\right|\right|_2
\end{align*}

For the second inequality, it is enough to prove the result when $W(g) = 
Z(g)$ or $W(g)=Z(g)^{2}$, and the general case will follow by linearity. We show this for $W(g)=Z(g)^2$. In this case we write from Proposition \ref{prop:consistency}
\begin{align*}
S^{\ell_0}(\mathbf{x}) =& \mathbf{Cov}_{\mathbf{x}}\left[Z(g)^2, \mathbbm{1}_{[g^{\ell_0}=+1]}\right]
\\
=& \frac{1}{L^2}\sum\limits_{\ell_1,\ell_2\in [L]} 
\mathbf{Cov}_{\mathbf{x}}\left[g^{\ell_1}g^{\ell_2}, \mathbbm{1}_{[g^{\ell_0}=+1]}\right]\\
=& \frac{1}{L^2}\sum\limits_{\ell_1,\ell_2 \in [L]} 
\mathbf{x}\left[g^{\ell_1}g^{\ell_2}\mathbbm{1}_{[g^{\ell_0}=+1]}\right]
- \mathbf{x}\left[g^{\ell_1}g^{\ell_2}\right]p^{\ell_0}(\mathbf{x})
\end{align*}
The analog holds for $S^{\ell_0}(\pi(\mathbf{x}))$. It follows
\begin{multline*}
\left|S^{\ell_0}(\mathbf{x}) - S^{\ell_0}(\pi(\mathbf{x}))\right|\\
\leq \frac{1}{L^2}\sum\limits_{\ell_1,\ell_2\in [L]} 
\Big|\mathbf{x}\left[g^{\ell_1}g^{\ell_2}\mathbbm{1}_{[g^{\ell_0}=+1]}\right]
- \pi(\mathbf{x})\left[g^{\ell_1}g^{\ell_2}\mathbbm{1}_{[g^{\ell_0}=+1]}\right]\Big|\;
 +\; p^{\ell_0}(\mathbf{x})\times\Big|\mathbf{x}[g^{\ell_1}g^{\ell_2}] - 
\pi(\mathbf{x})[g^{\ell_1}g^{\ell_2}] \Big|
\end{multline*}
We may remove the summmand corresponding to $\ell_1=\ell_2=\ell_0$, because $\mathbf{x}$ and $\pi(\mathbf{x})$ have the same marginals on $\ell_0$. We conclude by rewriting this
$$
\left|S^{\ell_0}(\mathbf{x}) - S^{\ell_0}(\pi(\mathbf{x}))\right|
\leq 
\frac{2}{L^2}
\sum\limits_{\substack{A\subset [L]\smallsetminus\{\ell_0\}\\1\leq \#A\leq 2}} 
\left|\left|\mathbf{x}^{\{\ell_0\}\cup A } - 
\pi(\mathbf{x}^{\{\ell_0\}\cup A})\right|\right|_1
$$
Similar calculations when $W(g)=Z(g)$ yield
$$
\left|S^{\ell_0}(\mathbf{x}) - S^{\ell_0}(\pi(\mathbf{x}))\right|
\leq 
\frac{2}{L}
\sum\limits_{\ell_1\in[L]\smallsetminus \{\ell_0\}} 
\left|\left|\mathbf{x}^{\{\ell_0,\ell_1\} } - 
\pi(\mathbf{x}^{\{\ell_0,\ell_1\}})\right|\right|_1
$$
We conclude using the equivalence of the $L^1$ and $L^2$ norms in $\RR^{\square_{A}}$. 
\end{proof}

\subsection{Eigenvalues of the linearized recombinator \label{sec:eigenrec}}
The goal of this section is to obtain some properties of the jacobian of the recombinator $\nabla \hat{R}^A$, which will allow us to find a lower bound for the contribution of recombination to the dynamics of $\mathbf{X}_t^A$. Because $\hat{R}^A$ is the analog of $R$, we will simplify our proofs without loss of generality by assuming $A=[L]$. The following Lemma motivates the study of the jacobian of the recombinator by relating it to the recombinator itself.
\begin{lem}\label{lem:nablaFI}
    $$
    \forall \mathbf{x}\in\XX^{[L]},\qquad
    \nabla R(\mathbf{x})(\mathbf{x}-\pi(\mathbf{x})) - R(\mathbf{x})
    = \sum\limits_{\emptyset\subsetneq\mathcal{I}\subsetneq [L]}\nu(\mathcal{I}) (\mathbf{x}-\pi(\mathbf{x}))^{\mathcal{I}}\otimes (\mathbf{x}-\pi(\mathbf{x}))^{\mathcal{I}^{c}}
    $$
In particular, if $\#A = 3$ then
\begin{equation}\label{eq:linearizeR}
\forall \mathbf{x}\in\XX^{A},\qquad 
\hat{R}^A(\mathbf{x})=\nabla \hat{R}^A(\mathbf{x})(\mathbf{x}-\pi(\mathbf{x}))
\end{equation}
\end{lem}
\begin{proof}
A simple computation shows that for any $\mathbf{h}\in\RR^{\square_{[L]}}$
\begin{equation}\label{eq:nablaFI}
\nabla R(\mathbf{x})\mathbf{h} = \sum\limits_{\emptyset\subsetneq\mathcal{I}\subsetneq [L]}\nu(\mathcal{I}) (\mathbf{x}^{\mathcal{I}}\otimes \mathbf{h}^{\mathcal{I}^c} 
+ \mathbf{h}^{\mathcal{I}}\otimes \mathbf{x}^{{\mathcal I}^c} - \mathbf{h})
\end{equation}
And thus
$$
\nabla R(\mathbf{x})\mathbf{h} - R({\bf x})
= \sum\limits_{\emptyset\subsetneq\mathcal{I}\subsetneq [L]}\nu(\mathcal{I}) \left( \mathbf{x}^{\mathcal{I}}\otimes \mathbf{h}^{{\mathcal I}^{c}}
+ (\mathbf{h}-\mathbf{x})^{\mathcal{I}}\otimes \mathbf{x}^{{\mathcal I}^{c}} - (\mathbf{h}-\mathbf{x}) \right)
$$
Applying this to 
$
\mathbf{h}=\mathbf{x}-\pi(\mathbf{x})$ yields
\begin{align*}
\nabla R({\bf x})(\mathbf{x}-\pi(\mathbf{x})) - R({\bf x})
=& \sum\limits_{\emptyset\subsetneq\mathcal{I}\subsetneq [L]}\nu(\mathcal{I}) (\mathbf{x}^{\mathcal{I}}\otimes (\mathbf{x}-\pi(\mathbf{x}))^{{\mathcal I}^{c}}
-\pi(\mathbf{x})^{\mathcal{I}}\otimes \mathbf{x}^{{\mathcal I}^{c}} + \pi(\mathbf{x}) )\\
=& \sum\limits_{\emptyset\subsetneq\mathcal{I}\subsetneq [L]}\nu(\mathcal{I}) ( \mathbf{x}^{\mathcal{I}}\otimes (\mathbf{x}-\pi(\mathbf{x}))^{{\mathcal I}^{c}}
-\pi(\mathbf{x})^{\mathcal{I}}\otimes (\mathbf{x}-\pi(\mathbf{x}))^{{\mathcal I}^{c}} )\\
=& \sum\limits_{\emptyset\subsetneq\mathcal{I}\subsetneq [L]}\nu(\mathcal{I}) (\mathbf{x}-\pi(\mathbf{x}))^{\mathcal{I}}\otimes (\mathbf{x}-\pi(\mathbf{x}))^{{\mathcal I}^{c}}
\end{align*}
where in the second equality we used $\pi(\mathbf{x}) = \pi(\mathbf{x})^{\mathcal{I}}\otimes \pi(\mathbf{x})^{{\mathcal I}^{c}}$.

Because $\hat{R}^A$ is the analog of $R$ on $\XX^A$, we obtain
\begin{equation}\label{eq:tmplinearizeR}
    \forall \mathbf{x}\in\XX^A,\qquad
    \nabla \hat{R}^A(\mathbf{x})(\mathbf{x}-\pi(\mathbf{x})) - \hat{R}^A(\mathbf{x})
    = \sum\limits_{\emptyset\subsetneq\mathcal{I}\subsetneq A}\nu^A(\mathcal{I}) (\mathbf{x}-\pi(\mathbf{x}))^{\mathcal{I}}\otimes (\mathbf{x}-\pi(\mathbf{x}))^{A\smallsetminus \mathcal{I}}
\end{equation}
Recall from (\ref{eq:defpi}) that for any $\ell\in A$,
$$\mathbf{x}^{\{\ell\}}=\pi(\mathbf{x})^{\{\ell\}}$$
It follows that whenever $\#\mathcal{I}\in\{1,2\}$, we necessarily have
$$
(\mathbf{x}-\pi(\mathbf{x}))^{\mathcal{I}}\otimes (\mathbf{x}-\pi(\mathbf{x}))^{A\smallsetminus \mathcal{I}} = 0
$$
This yields (\ref{eq:linearizeR}).
\end{proof}

Let us define the following quantities
\begin{equation}\label{eq:defbeta}
\forall \mathcal{I}\subseteq [L],\mathcal{I}\neq\emptyset,\qquad    \beta_{\mathcal{I}} := \sum_{\emptyset \subsetneq \mathcal{K}\subsetneq \mathcal{I}} \nu^{\mathcal{I}}(\mathcal{K})
= 1-\nu^{\mathcal{I}}(\mathcal{I})-\nu^{\mathcal{I}}(\emptyset) = 1-2\nu^{\mathcal{I}}(\mathcal{I})
\end{equation}
We call $\beta_{\mathcal{I}}$ the probability that there is a recombination within $\mathcal{I}$. Also define $\beta_{\emptyset}:=-\beta_{[L]}$. We note the following order property
\begin{equation}\label{eq:orderproperty0}
\forall \mathcal{J}\subseteq \mathcal{I}\subseteq[L],\qquad \beta_{\mathcal{I}}\geq \beta_{\mathcal{J}}.
\end{equation}

For $\mathcal{I}\subseteq [L]$, we define the $\mathcal{I}$-\textbf{linkage vector} with
\begin{equation}\label{eq:Ilinkagevect}
\mathbf{w}_{\mathcal{I}} := \left(2^{-\frac{L}{2}}
\prod\limits_{\ell\in \mathcal{I}} \gamma^\ell\right)_{\gamma\in \square_{[L]}} \in \RR^{\square_{[L]}}
\end{equation}
We start with computational properties of the linkage vectors.
\begin{lem}\label{lem:wbasisA}
The $(\mathbf{w}_{\mathcal{I}})_{\mathcal{I}\subseteq [L]}$ form an orthonormal basis of $\RR^{\square_{[L]}}$ for the usual scalar product, which will be denoted $\braket{\cdot,\cdot}$.
Furthermore, we have
\begin{align}
    \forall \mathcal{J}\subseteq [L],\qquad\qquad \mathbf{w}_{\mathcal{I}} =& 2^{-\frac{L}{2}}\mathbf{w}_{\mathcal{I}\cap \mathcal{J}}^{\mathcal{J}} 
    \otimes \mathbf{w}_{\mathcal{I}\cap \mathcal{J}^c}^{\mathcal{J}^c}\label{eq:productw}\\
    \forall \mathcal{J}, \mathcal{I}\subset [L],\forall \mathbf{x}\in\RR^{\square_{[L]}},\qquad
    \braket{\mathbf{w}_{\mathcal{I}}^{\mathcal{J}},\mathbf{x}^{\mathcal{J}}} =& 
    2^{L-\#\mathcal{J}}\braket{\mathbf{w}_{\mathcal{I}},\mathbf{x}}\mathbbm{1}_{[\mathcal{I}\subseteq \mathcal{J}]}\label{eq:unmarginalw}
\end{align}
\end{lem} 
\begin{proof}
We compute
\begin{align*}
\braket{\mathbf{w}_{\mathcal{I}},\mathbf{w}_{\mathcal{J}}} 
=& 2^{-L} \sum\limits_{\gamma\in \square_{[L]}}
\prod\limits_{\ell\in \mathcal{I}} \gamma^\ell
\prod\limits_{\ell\in \mathcal{J}} \gamma^\ell\\
=& 2^{-L} \sum\limits_{\gamma\in \square_{[L]}}
\prod\limits_{\ell\in \mathcal{I} \oplus \mathcal{J}} \gamma^\ell 
\end{align*}
where we write $\mathcal{I}\oplus \mathcal{J} := (\mathcal{I}\cup \mathcal{J})\smallsetminus (\mathcal{I}\cap \mathcal{J})$. If $\mathcal{I}=\mathcal{J}$, this yields $1$. Otherwise, we can find $\ell_0\in \mathcal{I}\oplus\mathcal{J}$. Then we can cancel out the term $\gamma$ of the sum with the term $\gamma_{-\ell_0}$, which we define to be $\gamma$ with the $\ell_0-$th coordinate flipped.  This yields that $(\mathbf{w}_{\mathcal{I}})_{\mathcal{I}\subseteq [L]}$ is an orthonormal basis.

Let us observe that whenever $\gamma\in\square_{\mathcal{J}}, \mathcal{I}\subseteq \mathcal{J}$ we have
$$
\mathbf{w}_{\mathcal{I}}^{\mathcal{J}}(\gamma) = 
\sum_{\substack{\hat{\gamma}\in\square_{[L]}\\ \hat{\gamma}|_{\mathcal{J}} = \gamma}}
2^{-\frac{L}{2}} \prod_{\ell\in\mathcal{I}} \gamma^\ell
=  2^{-\frac{L}{2}} \times 2^{L-\#\mathcal{J}} \prod_{\ell\in\mathcal{I}} \gamma^\ell 
$$
We then get (\ref{eq:productw}) with
$$
\mathbf{w}_{\mathcal{I}\cap \mathcal{J}}^{\mathcal{J}} \otimes \mathbf{w}_{\mathcal{I}\cap \mathcal{J}^c}^{\mathcal{J}^c}(\gamma)
\;=\; 2^{\frac{L}{2}-\#\mathcal{J}}  \prod_{\ell\in\mathcal{I}\cap \mathcal{J}} \gamma^\ell\;\times\;
2^{\frac{L}{2}-\#\mathcal{J}^c}  \prod_{\ell\in\mathcal{I}\cap \mathcal{J}^c} \gamma^\ell
\;=\; 2^{\frac{L}{2}} \mathbf{w}_{\mathcal{I}}
$$
To get (\ref{eq:unmarginalw}) when $\mathcal{I}\subseteq \mathcal{J}$ we write
$$
\left<\mathbf{w}_{\mathcal{I}}^\mathcal{J}\;,\;\mathbf{x}^{\mathcal{J}}\right>
=
2^{\frac{L}{2}-\#\mathcal{J}}\sum\limits_{\gamma\in\square_{[L]}} \left(\prod\limits_{\ell\in\mathcal{I}}\gamma^\ell\right)x(\gamma)
= 2^{L-\#\mathcal{J}}\left<\mathbf{w}_{\mathcal{I}}\;,\;\mathbf{x}\right>
$$
When $\mathcal{I}\nsubseteq \mathcal{J}$, then we can find $\ell_0\in\mathcal{I}\smallsetminus\mathcal{J}$. We then write
$$
\mathbf{w}_{\mathcal{I}}^{\mathcal{J}}(\gamma) = 
\sum_{\substack{\hat{\gamma}\in\square_{[L]}\\ \hat{\gamma}|_{\mathcal{J}} = \gamma}}
2^{-\frac{L}{2}} \prod_{\ell\in\mathcal{I}} \gamma^\ell = 0
$$
where we cancelled out the $\hat{\gamma}$ term of the sum with the $\hat{\gamma}_{-\ell_0}$ term, that is, $\hat{\gamma}$ with the $\ell_0-$th coordinate flipped.
\end{proof}

Because $\hat{R}^A$ is the analog of $R$ on $\RR^{\square_A}$, the computations are the same. We will in fact compute $\nabla R$ explicitely in the next Theorem.
\begin{thm}\label{thm:eigenvalueR}
We have for any $\mathcal{I},\mathcal{J}\subseteq [L]$.
\begin{equation*}
\braket{\mathbf{w}_{\mathcal{I}},\nabla R(\mathbf{x})\mathbf{w}_{\mathcal{J}}}
=
    -\mathbbm{1}_{[\mathcal{I}=\mathcal{J}]}\beta_{\mathcal{I}} +
    \mathbbm{1}_{[\mathcal{J}\subsetneq \mathcal{I}]}
    \sum\limits_{\emptyset\subsetneq\mathcal{K}\subsetneq [L]}\nu(\mathcal{K})
2^{1+\frac{L}{2}}\left<\mathbf{w}_{\mathcal{I}\cap \mathcal{K}}\;,\;\mathbf{x}\right>
\mathbbm{1}_{[\mathcal{I}\cap \mathcal{K}^c = \mathcal{J}]}
\end{equation*}
In particular, we have
\begin{equation}\label{eq:triangularnablaR}
\forall\mathcal{J}\nsubseteq \mathcal{I},\qquad \braket{\mathbf{w}_{\mathcal{I}},\nabla R(\mathbf{x})\mathbf{w}_{\mathcal{J}}}=0
\end{equation}
\end{thm}
\begin{rem}
A well-known method to handle the recombinator is Haldane linearization \cite{BaakeHaldane_linearization}. This method relies on considering linkage, not on subsets $\mathcal{I}\subseteq A$, but rather on partitions of $A$. This approach may prove necessary if we want to control linkage on subsets $A$ with arbitrary size, but will not be needed here.
\end{rem}
\begin{proof}
According to (\ref{eq:nablaFI}) (replacing $[L]$ by the set $A$),
\begin{align*}
\braket{\mathbf{w}_{\mathcal{I}},\nabla R(\mathbf{x})\mathbf{w}_{\mathcal{J}}}
=&\sum\limits_{\emptyset\subsetneq\mathcal{K}\subsetneq [L]}\nu(\mathcal{K})
\left<\mathbf{w}_{\mathcal{I}}\;,\;\mathbf{x}^{\mathcal{K}}\otimes \mathbf{w}_{\mathcal{J}}^{\mathcal{K}^{c}} 
+ \mathbf{w}_{\mathcal{J}}^{\mathcal{K}}\otimes \mathbf{x}^{\mathcal{K}^{c}} 
- \mathbf{w}_{\mathcal{J}}\right>\\
=&\sum\limits_{\emptyset\subsetneq\mathcal{K}\subsetneq [L]}\nu(\mathcal{K})
\left(\left<\mathbf{w}_{\mathcal{I}}\;,\;\mathbf{x}^{\mathcal{K}}\otimes \mathbf{w}_{\mathcal{J}}^{\mathcal{K}^c}\right>
+ \left<\mathbf{w}_{\mathcal{I}}\;,\;\mathbf{w}_{\mathcal{J}}^{\mathcal{K}}\otimes \mathbf{x}^{\mathcal{K}^c}\right>
- \mathbbm{1}_{[\mathcal{I}=\mathcal{J}]}\right)
\end{align*}
 Using $\nu(\mathcal{K}) = \nu(\mathcal{K}^c)$, we can rewrite this
\begin{equation}\label{eq:tmptriangR}
\braket{\mathbf{w}_{\mathcal{I}},\nabla R(\mathbf{x})\mathbf{w}_{\mathcal{J}}} 
=\sum\limits_{\emptyset\subsetneq\mathcal{K}\subsetneq [L]}\nu(\mathcal{K})
\left(2\left<\mathbf{w}_{\mathcal{I}}\;,\;\mathbf{x}^{\mathcal{K}}\otimes \mathbf{w}_{\mathcal{J}}^{\mathcal{K}^c}\right>
- \mathbbm{1}_{[\mathcal{I}=\mathcal{J}]}\right)
\end{equation}
Using (\ref{eq:productw}) we get
\begin{align*}
\left<\mathbf{w}_{\mathcal{I}}\;,\;\mathbf{x}^{\mathcal{K}}\otimes \mathbf{w}_{\mathcal{J}}^{\mathcal{K}^c}\right>
=& 2^{-\frac{L}{2}}\left<\mathbf{w}_{\mathcal{I}\cap \mathcal{K}}^\mathcal{K}\;,\;\mathbf{x}^{\mathcal{K}}\right>
\left<\mathbf{w}_{\mathcal{I}\cap \mathcal{K}^c}^{\mathcal{K}^c},\mathbf{w}_{\mathcal{J}}^{\mathcal{K}^c}\right>\\
=& 2^{\frac{L}{2}}\left<\mathbf{w}_{\mathcal{I}\cap \mathcal{K}}\;,\;\mathbf{x}\right>
\mathbbm{1}_{[\mathcal{I}\cap \mathcal{K}^c = \mathcal{J}]}
\end{align*}
where in the second equality we used (\ref{eq:unmarginalw}) twice. We thus obtain from (\ref{eq:tmptriangR})
\begin{align}
\braket{\mathbf{w}_{\mathcal{I}},\nabla R(\mathbf{x})\mathbf{w}_{\mathcal{J}}} 
=&\sum\limits_{\emptyset\subsetneq\mathcal{K}\subsetneq [L]}\nu(\mathcal{K})
\left(2^{1+\frac{L}{2}}\left<\mathbf{w}_{\mathcal{I}\cap \mathcal{K}}\;,\;\mathbf{x}\right>
\mathbbm{1}_{[\mathcal{I}\cap \mathcal{K}^c = \mathcal{J}]}
- \mathbbm{1}_{[\mathcal{I}=\mathcal{J}]}\right)\notag\\
=&\sum\limits_{\emptyset\subsetneq\mathcal{K}\subsetneq [L]}\nu(\mathcal{K})
2^{1+\frac{L}{2}}\left<\mathbf{w}_{\mathcal{I}\cap \mathcal{K}}\;,\;\mathbf{x}\right>
\mathbbm{1}_{[\mathcal{I}\cap \mathcal{K}^c = \mathcal{J}]}
\quad-\; \mathbbm{1}_{[\mathcal{I}=\mathcal{J}]}\beta_{[L]}\label{eq:tmp(wI,wJ)}
\end{align}
where in the last equality we used
$$
\sum\limits_{\emptyset\subsetneq\mathcal{K}\subsetneq [L]} \nu(\mathcal{K}) = \beta_{[L]}
$$
Notice that the sum in (\ref{eq:tmp(wI,wJ)}) can only be nonzero if $\mathcal{J}\subseteq \mathcal{I}$. When $\mathcal{J}\neq \mathcal{I}$, we get the result. 
When $\mathcal{I}=\mathcal{J}$ we write
\begin{align*}
\braket{\mathbf{w}_{\mathcal{I}},\nabla R(\mathbf{x})\mathbf{w}_{\mathcal{I}}} 
=& \sum\limits_{\emptyset\subsetneq\mathcal{K}\subsetneq [L]}\nu(\mathcal{K})
2^{1+\frac{L}{2}}\left<\mathbf{w}_{\emptyset}\;,\;\mathbf{x}\right>
\mathbbm{1}_{[\mathcal{I}\cap \mathcal{K}^c = \mathcal{I}]}
\quad-\; \beta_{[L]}\\
=& \sum\limits_{\emptyset\subsetneq\mathcal{K}\subsetneq [L]}2\nu(\mathcal{K})
\mathbbm{1}_{[\mathcal{I}\subseteq \mathcal{K}^c]}
\quad-\; \beta_{[L]}
\end{align*}
where we used $\braket{\mathbf{w}_\emptyset,\mathbf{x}}=2^{-\frac{L}{2}}$. If $\mathcal{I}=\emptyset$, the sum is equal to $2\beta_{[L]}$ and we get
$$
\braket{\mathbf{w}_{\emptyset},\nabla R(\mathbf{x})\mathbf{w}_{\emptyset}}
= \beta_{[L]} = -\beta_{\emptyset}
$$
If $\mathcal{I}\neq\emptyset$, we can extend the sum to $\mathcal{K}\in\{\emptyset,[L]\}$ and write
\begin{align*}
\sum\limits_{\emptyset\subsetneq\mathcal{K}\subsetneq [L]}2\nu(\mathcal{K})\mathbbm{1}_{[\mathcal{I}\subseteq \mathcal{K}^c]} =&
\sum\limits_{\emptyset\subseteq\mathcal{K}\subseteq [L]}2\nu(\mathcal{K})\mathbbm{1}_{[\mathcal{I}\subseteq \mathcal{K}^c]}
\quad - \;2\nu(\emptyset)\\
=&
2\nu^{\mathcal{I}}(\mathcal{I})
 - (1-\beta_{[L]})
\end{align*}
It follows
\begin{align*}
\braket{\mathbf{w}_{\mathcal{I}},\nabla R(\mathbf{x})\mathbf{w}_{\mathcal{I}}} 
=& 2\nu^{\mathcal{I}}(\mathcal{I}) - (1-\beta_{[L]}) - \beta_{[L]}\\
=& 2\nu^{\mathcal{I}}(\mathcal{I}) - 1 \\
=& -\beta_{\mathcal{I}}
\end{align*}
\end{proof}

\subsection{Controlling linkage disequilibrium over a small subset \label{sec:controlL2itself}}
For any $A\subseteq [L]$ with $\#A\geq2$, recall from (\ref{eq:defrA})
$$
    r_A:= \min_{\substack{\ell_1,\ell_2\in A\\ \ell_1\neq\ell_2}} r_{\{\ell_1,\ell_2\}}
$$
where $r_{\{\ell_1,\ell_2\}}$ is the probability of a recombination event between $\ell_1$ and $\ell_2$, and in particular from (\ref{eq:defbeta}) $r_{\{\ell_1,\ell_2\}}=\beta_{\{\ell_1,\ell_2\}}$. From (\ref{eq:orderproperty0}) we have for any $\mathcal{J}\subseteq A$
\begin{equation}\label{eq:orderrAbeta1}
r_A \leq \beta_{\mathcal{J}} \leq 1.
\end{equation}

The goal of this subsection is to prove
\begin{prop}\label{prop:controllink}
Let $T>0$.
Assume that  $\eta^2 := \frac{1}{\rho r_A}\leq 1$ and $\rho \geq e$. There exists a constant $C$ independent of $(A,L)$ such that 
for every $\varepsilon \in[0,T]$, 
$$
\forall A\subseteq [L],\#A\leq 3,\qquad \EE\left[\sup\limits_{t\in [\varepsilon,T]} ||\mathbf{X}_t^A - \pi(\mathbf{X}_t^A)||_2\right] \leq C\left( \frac{1}{\rho r_A\varepsilon} 
+\frac{L}{\rho r_A} + \sqrt{\frac{\ln(\rho)}{\rho r_A}}\right).
$$
\end{prop}
We will prove the result for $\#A=3$. We start with two Lemmas.
\begin{lem}\label{lem:linearizeR}
Consider $A\subset [L]$ with $\#A\leq 3$. Then
\begin{equation}\label{eq:calculid-piR}
\forall \mathbf{x}\in\XX^A,\qquad
\left(\mbox{Id}-\nabla \pi(\mathbf{x})\right)\hat{R}^A(\mathbf{x}) 
= \nabla\hat{R}^A(\mathbf{x})(\mathbf{x}-\pi(\mathbf{x}))
\end{equation}
where we abusively write $\nabla \pi$ for the $\square_A\times \square_A$ jacobian of $\pi$, seen as a $\RR^{\square_{A}}-$valued function on $\XX^A$.
\end{lem}
\begin{proof}
Because of (\ref{eq:linearizeR}) in Lemma \ref{lem:nablaFI}, we only need to show
\begin{equation}\label{eq:nablapiReq0}
\forall \mathbf{x}\in\XX^A,\qquad
\nabla \pi(\mathbf{x}) \hat{R}^A(\mathbf{x}) = 0
\end{equation}
Fix $\mathbf{x}_0\in \XX^A$. We define $(\mathbf{x}_t)_{t\geq 0}$ as the solution to
$$
\frac{\dd}{\dd t}\mathbf{x}_t = \hat{R}^A(\mathbf{x}_t)
$$
with initial condition $\mathbf{x}_0$. By Lemma \ref{lem:defnolink}, we have
$$\forall t\geq 0, \qquad \pi(\mathbf{x}_t) = \pi(\mathbf{x}_0).$$
Taking the time derivative of $\pi(\mathbf{x}_t)$, we get
$$
\mathbf{0} = \frac{\dd}{\dd t}\pi(\mathbf{x}_t) = \nabla \pi(\mathbf{x}_t) \hat{R}^A(\mathbf{x}_t)
$$
Evaluating the right-hand side at $t=0$ gives (\ref{eq:nablapiReq0}).
\end{proof}

\begin{lem}\label{lem:controlOU}
Consider a continuous martingale $M_t$ with quadratic variation uniformly bounded by $C_0>0$. Then we can find a universal constant $C(C_0)$ such that for any $\lambda\geq 1,T>0$
$$
\EE\left[\sup_{t\in[0,T]} \left|\int_0^t e^{-\lambda(t-u)} \dd M_u\right|\right] \leq C\sqrt{\frac{1+\ln_+(2\lambda T)}{\lambda}}
$$
with $\ln_+(x)=(\ln(x))\vee 0$.
\end{lem}
\begin{proof}
Let us write
$$
Z_t:= \int_0^t e^{\lambda u} \dd M_u \qquad;\qquad Y_t:= \int_0^t e^{-\lambda(t-u)} \dd M_u
$$
such that
\begin{equation}\label{eq:tmpornsteind}
Y_t = e^{-\lambda t}Z_t
\end{equation}
The martingale $(Z_t)_{t\in[0,T]}$ has quadratic variation 
$$
\braket{Z}_t^{QV} = \int_0^t e^{2\lambda u} \dd \braket{M}_u^{QV} \leq C_0\frac{e^{2\lambda t}}{2\lambda}
$$
Let us define the iterated logarithm
$$
\ln_{(2)}(x) := \ln_+(\ln_+(x)).
$$
Since $\lambda\geq 1$
\begin{align*}
\ln_{(2)}\left(\frac{e^{2\lambda t}}{2\lambda}\right) 
 \leq& \ln_+(2\lambda T).
\end{align*}
Therefore 
$$
F(t):=\braket{Z}_t^{QV} \left(1+\ln_{(2)}(\braket{Z}_t^{QV})\right) \leq C_0\frac{e^{2\lambda t}}{2\lambda} \left(1 + \ln_{(2)}(C_0) + \ln_+(2\lambda T)\right)
$$
It follows from (\ref{eq:tmpornsteind})
\begin{align*}
\sup_{t\in [0,T]}|Y_t| =& \sup_{t\in [0,T]} e^{-\lambda t} |Z_t|\\
=& \sup_{t\in [0,T]} e^{-\lambda t}\sqrt{F(t)} \times \frac{|Z_t|}{\sqrt{F(t)}}\\
\leq& \sqrt{C_0}\sqrt{\frac{1+\ln_+(C_0)+\ln_+(2\lambda T)}{2\lambda}}\sup_{t\in [0,T]} \frac{|Z_t|}{\sqrt{F(t)}}
\end{align*}
We show in Appendix \ref{sec:apLIL} the existence of a universal constant $C>0$ (only depending on $C_0$) such that
$$
\EE\left[\sup_{t\in [0,T]} \frac{Z_t}{\sqrt{F(t)}}\right] \leq C.
$$
This completes the proof of the lemma.
\end{proof}

\begin{proof}[Proof of Proposition \ref{prop:controllink}]
For $\mathcal{I}\subseteq A$, define by analogy with (\ref{eq:Ilinkagevect})
$$\hat{\mathbf{w}}_{\mathcal{I}}^A = \left(2^{-\#A/2}\prod_{\ell\in A} \gamma^\ell\right)_{\gamma\in\square_A} $$
One may check that $\hat{\mathbf{w}}_{\mathcal{I}}^A = 2^{(L-\#A)/2}\mathbf{w}_{\mathcal{I}}^A$.
By the analog of Lemma \ref{lem:wbasisA}, $(\hat{\mathbf{w}}_{\mathcal{I}}^A)_{\mathcal{I}\subseteq A}$ is an orthonormal basis of $\RR^{\square_A}$ and we have for any $\mathbf{y}\in\RR^{\square_{A}}$
$$
||\mathbf{y}||_2^2 = \sum_{\mathcal{I}\subseteq A} \braket{\hat{\mathbf{w}}_{\mathcal{I}}^A,\mathbf{y}}^2.
$$
Therefore, to get the result we only need to control, for all $\mathcal{I}\subseteq A$,
$$Y_t^{\mathcal{I}} := \braket{\hat{\mathbf{w}}_{\mathcal{I}}^A,\mathbf{X}_t^A - \pi(\mathbf{X}_t^A)}.$$
Note that because $\hat{\mathbf{w}}_{\mathcal{I}}^A$ and $\mathbf{X}_t^A - \pi(\mathbf{X}_t^A)$ have coefficients bounded by 1, and $\RR^A$ has dimension at most $2^3$, then $\braket{\hat{\mathbf{w}}_{\mathcal{I}}^A,\mathbf{X}_t^A}$ and $Y_t^{\mathcal{I}}$ are uniformly bounded by a constant $C$ independent of $(L,A)$.

Recall from the definition of $\pi$ in (\ref{eq:defpi}) that for any $\mathbf{x}\in\XX^A$
\begin{equation}\label{eq:x-pi(x)ell}
Y_t^\emptyset = \braket{\hat{\mathbf{w}}_{\emptyset}^A,\mathbf{x}_t - \pi(\mathbf{x}_t)}
= 0
\quad\qquad;\quad\qquad
\forall \ell\in A,\qquad
Y_t^{\{\ell\}}=\braket{\hat{\mathbf{w}}_{\{\ell\}}^A,\mathbf{x}_t - \pi(\mathbf{x}_t)}
= 0
\end{equation}
It remains to consider $\#\mathcal{I}\in \{2,3\}$. We apply It\^o's formula to write
$$
    \dd \left(\mathbf{X}_u^A - \pi(\mathbf{X}_u^A)\right) = \left(\mbox{Id}-\nabla\pi(\mathbf{X}_u^A)\right)\dd \mathbf{X}_u^A
    - \sum_{\gamma_1,\gamma_2\in\square_A} \partial_{\gamma_1,\gamma_2}\pi(\mathbf{X}_u^A)\;
    \dd \braket{X^A(\gamma_1),X^A(\gamma_2)}_u^{QV}
$$
Using Proposition \ref{prop:consistency} this can be rewritten
$$
\dd \left(\mathbf{X}_u^A - \pi(\mathbf{X}_u^A)\right) = \rho\left(\mbox{Id}-\nabla\pi(\mathbf{X}_u^A)\right)\hat{R}^A(\mathbf{X}_u^A)\dd t + \mathbf{F}_u\dd u
+ \dd \mathbf{M}_u
$$
where
\begin{align*}
    \mathbf{F}_t :=& \left(\mbox{Id}-\nabla\pi(\mathbf{X}_t^A)\right)\left(\hat{\Theta}^{A}(\mathbf{X}_t^A) + LS^{A}(\mathbf{X}_t) \right)
    - \sum_{\gamma_1,\gamma_2\in\square_A} \partial_{\gamma_1,\gamma_2}\pi(\mathbf{X}_t^A) \frac{\dd}{\dd t} \braket{X^A(\gamma_1),X^A(\gamma_2)}_t^{QV}\;
\\
    \dd\mathbf{M}_t :=& \left(\mbox{Id}-\nabla\pi(\mathbf{X}_t^A)\right)
                        \hat{\Sigma}^{A}(\mathbf{X}_t^A)\dd\hat{\mathbf{B}}_t^A.
\end{align*}
In particular, $\mathbf{F}_t$ is smaller than $CL$  and $\mathbf{M}_t$ is a continuous martinale with quadratic variation uniformly smaller than $C$ for some constant $C>0$ independent of $(A,L)$.
Finally, Lemma \ref{lem:linearizeR} yields for $\mathcal{I}\subseteq A$
\begin{equation}\label{eq:tmpwx-pi(x)}
    \dd Y_u^{\mathcal{I}} = 
    \rho\left<\hat{\mathbf{w}}_{\mathcal{I}}^A,\nabla\hat{R}^A(\mathbf{X}_u^A)
    \left(\mathbf{X}_u^A - \pi(\mathbf{X}_u^A)\right)\right>\dd u
    + \left<\hat{\mathbf{w}}_{\mathcal{I}}^A,\mathbf{F}_u\right>\dd u 
    + \left<\hat{\mathbf{w}}_{\mathcal{I}}^A,\dd \mathbf{M}_u\right>
\end{equation}
\bigskip

\textbf{Case} $\#\mathcal{I}=2$:\\
Theorem \ref{thm:eigenvalueR} and (\ref{eq:x-pi(x)ell}) imply
$$
    \left<\hat{\mathbf{w}}_{\mathcal{I}}^A,\nabla\hat{R}^A(\mathbf{X}_u^A)
    \left(\mathbf{X}_u^A - \pi(\mathbf{X}_u^A)\right)\right>
    = -\beta_{\mathcal{I}} \left<\hat{\mathbf{w}}_{\mathcal{I}}^A,\;\mathbf{X}_u^A - \pi(\mathbf{X}_u^A)\right>
    = -\beta_{\mathcal{I}} Y_u^{\mathcal{I}}
$$
Therefore (\ref{eq:tmpwx-pi(x)}) becomes
$$
    \dd Y_u^{\mathcal{I}} = 
    -\rho \beta_{\mathcal{I}} Y_u^{\mathcal{I}}\dd u
    + \left<\hat{\mathbf{w}}_{\mathcal{I}}^A,\mathbf{F}_u\right>\dd u 
    + \left<\hat{\mathbf{w}}_{\mathcal{I}}^A,\dd \mathbf{M}_u\right>
$$
It can be checked that this implies
\begin{equation}\label{eq:solvedI2}
Y_t^{\mathcal{I}} = Y_0^{\mathcal{I}} e^{-\rho \beta_{\mathcal{I}}t} 
+ \int_0^t e^{-\rho \beta_{\mathcal{I}}(t-u)}\left<\hat{\mathbf{w}}_{\mathcal{I}}^A,\mathbf{F}_u\right>\dd u
+ \int_0^t e^{-\rho \beta_{\mathcal{I}}(t-u)}\left<\hat{\mathbf{w}}_{\mathcal{I}}^A,\dd \mathbf{M}_u\right>
\end{equation}

The fact that $Y_0^{\mathcal{I}}$ is uniformly bounded means the first term is smaller than $Ce^{-\rho \beta_{\mathcal{I}}\varepsilon}$ for any $t\in[\varepsilon,T]$.
We use $e^{-x}\leq \frac{1}{x}$ for $x\geq 0$ to write
\begin{equation*}
Y_0^{\mathcal{I}} e^{-\rho \beta_{\mathcal{I}}t}  \leq \frac{C}{\rho \beta_{\mathcal{I}} \varepsilon}
\leq \frac{C}{\rho r_A \varepsilon}
\end{equation*}
from (\ref{eq:orderrAbeta1}).
Because $\mathbf{F}_t$ is smaller than $CL$, the first integral is smaller than $C\frac{L}{\rho\beta_{\mathcal{I}}}$. For the second integral, we use Lemma \ref{lem:controlOU} applied with $\lambda = \rho \beta_{\mathcal{I}}$. Because $\rho \beta_{\mathcal{I}}\geq \rho r_A\geq 1$ by assumption, this Lemma yields a a bound on the second integral of
$$ C\sqrt{\frac{1+\ln_+(\rho \beta_{\mathcal{I}} 2T)}{\rho \beta_{\mathcal{I}}}}\leq C\sqrt{\frac{\ln(\rho)}{\rho r_A}}
.$$
using (\ref{eq:orderrAbeta1}) and $\rho \geq e$, for some constant $C$ independent of $(\varepsilon,A,L)$.
We thus obtain from (\ref{eq:solvedI2})
$$
\EE\left[\sup_{t\in[\varepsilon,T]}|Y_t^{\mathcal{I}}|\right] \leq C\left(\frac{1}{\rho r_A \varepsilon} 
+ \frac{L}{\rho r_A}
+\sqrt{\frac{\ln(\rho)}{\rho r_A}}\right)
$$
\bigskip

\textbf{Case} $\#\mathcal{I}=3$:\\
When $\mathcal{I}=A$, Theorem \ref{thm:eigenvalueR} and (\ref{eq:x-pi(x)ell}) imply
\begin{multline*}
    \left<\hat{\mathbf{w}}_A^A,\nabla\hat{R}^A(\mathbf{X}_u^A)
    \left(\mathbf{X}_u^A - \pi(\mathbf{X}_u^A)\right)\right>
    = -\beta_{A} \left<\hat{\mathbf{w}}_{A}^A,\;\mathbf{X}_u^A - \pi(\mathbf{X}_u^A)\right> \\
    + \sum_{\substack{\mathcal{J}\subsetneq A\\ \#\mathcal{J}=2}} \nu^A(\mathcal{J}) 2^{1+\frac{\# A}{2}}
    \left<\hat{\mathbf{w}}_{A\smallsetminus\mathcal{J}}^A,\;\mathbf{X}_u^A\right> 
    \left<\hat{\mathbf{w}}_{\mathcal{J}}^A,\;\mathbf{X}_u^A - \pi(\mathbf{X}_u^A)\right> 
\end{multline*}
Therefore, (\ref{eq:tmpwx-pi(x)}) becomes
$$
    \dd Y_t^{A} = 
    -\rho \beta_{A} Y_t^{A}\dd t
    + \rho  \sum_{\substack{\mathcal{J}\subsetneq A\\ \#\mathcal{J}=2}} \nu^A(\mathcal{J}) 2^{1+\frac{\# A}{2}}
    \left<\hat{\mathbf{w}}_{A\smallsetminus\mathcal{J}}^A,\;\mathbf{X}_t^A\right> 
    Y_t^{\mathcal{J}} \dd t
    + \left<\hat{\mathbf{w}}_{A}^A,\mathbf{F}_t\right>\dd t 
    + \left<\hat{\mathbf{w}}_{A}^A,\dd \mathbf{M}_t\right>
$$
which can be solved as
\begin{multline*}
Y_t^{A} = Y_0^{A} e^{-\rho \beta_{A}t} 
+ \rho  \sum_{\substack{\mathcal{J}\subsetneq A\\ \#\mathcal{J}=2}} \nu^A(\mathcal{J}) 2^{1+\frac{\# A}{2}}
   \int_0^t e^{-\rho \beta_A(t-u)}\left<\hat{\mathbf{w}}_{A\smallsetminus\mathcal{J}}^A,\;\mathbf{X}_u^A\right> 
    Y_u^{\mathcal{J}}\dd u\\
+ \int_0^t e^{-\rho \beta_{A}(t-u)}\left<\hat{\mathbf{w}}_{A}^A,\mathbf{F}_u\right>\dd u
+ \int_0^t e^{-\rho \beta_{A}(t-u)}\left<\hat{\mathbf{w}}_{A}^A,\dd \mathbf{M}_u\right>.
\end{multline*}
The first, third and fourth terms are handled as in the case $\#\mathcal{I}=2$. For the second term, let us define
\begin{align*}
    b_1(t) :=& \rho \beta_A\int_0^t e^{-\rho \beta_A(t-u)} \times 
    \left|Y_0^{\mathcal{J}}\right| e^{-\rho \beta_{\mathcal{J}}u} \dd u\\
    b_2(t) :=& \rho \beta_A\int_0^t e^{-\rho \beta_A(t-u)} \left( \int_0^u e^{-\rho \beta_{\mathcal{J}}(u-v)}
    \left|\left<\hat{\mathbf{w}}_{\mathcal{J}}^A,\mathbf{F}_v\right>\right|
    \dd v\right)\dd u\\
    b_3(t) :=& \rho \beta_A
    \int_0^t e^{-\rho \beta_A(t-u)} \left| \int_0^u e^{-\rho \beta_{\mathcal{J}}(u-v)}
    \left<\hat{\mathbf{w}}_{\mathcal{J}}^A,\dd \mathbf{M}_v\right>\right|\dd u
\end{align*}
Notice from (\ref{eq:defbeta}) that $\nu^A(\mathcal{J}) \leq \beta_{\mathcal{J}}$. This, along with 
the boundedness of $\left<\hat{\mathbf{w}}_{A\smallsetminus\mathcal{J}}^A,\;\mathbf{X}_u^A\right>$
  and (\ref{eq:solvedI2}) lets us write for $\mathcal{J}\subseteq A$ with $\#\mathcal{J}=2$
\begin{equation}\label{eq:b1b2b3}
\rho \nu^A(\mathcal{J})
   \int_0^t e^{-\rho \beta_A(t-u)}
   \left|\left<\hat{\mathbf{w}}_{A\smallsetminus\mathcal{J}}^A,\;\mathbf{X}_u^A\right> 
    Y_u^{\mathcal{J}}\right|\dd u
    \leq C(b_1(t) + b_2(t) + b_3(t))
\end{equation}
where $C$ is a constant independent of $(\varepsilon,A,L)$.
We now control $b_1$ by writing
\begin{align*}
b_1(t) =&\rho \beta_{\mathcal{J}}\int_0^t e^{-\rho \beta_A(t-u)} 
     e^{-\rho \beta_{\mathcal{J}}u} \dd u \times \left|Y_0^{\mathcal{J}}\right|
    +\rho(\beta_A-\beta_{\mathcal{J}}) \int_0^t e^{-\rho \beta_A(t-u)} e^{-\rho \beta_{\mathcal{J}}u} \dd u
    \times \left|Y_0^{\mathcal{J}}\right|\\
    =& \rho \beta_{\mathcal{J}} e^{-\rho \beta_{\mathcal{J}}t} \int_0^t e^{-\rho (\beta_A-\beta_{\mathcal{J}})(t-u)} 
      \dd u \times \left|Y_0^{\mathcal{J}}\right|
    + e^{-\rho \beta_{\mathcal{J}}t} \int_0^t \rho(\beta_A-\beta_{\mathcal{J}}) e^{-\rho (\beta_A-\beta_{\mathcal{J}})(t-u)} \dd u
    \times \left|Y_0^{\mathcal{J}}\right|
\end{align*}
Because of the order property (\ref{eq:orderproperty0}), $\beta_A\geq \beta_{\mathcal{J}}$. So the first integral is smaller than $t$, the second one is smaller than one. We thus obtain
$$
b_1(t) \leq C(1+\rho \beta_{\mathcal{J}}t) e^{-\rho\beta_{\mathcal{J}} t}
$$
We can use the inequality
$$
\forall x,h\geq0,\qquad
(x+h)e^{-(x+h)} \leq \frac{C}{x}
$$
with $x=\rho \beta_{\mathcal{J}}\varepsilon$ and $h=\rho \beta_{\mathcal{J}}(t-\varepsilon)$ to conclude
\begin{equation}\label{eq:b1control}
\sup_{t\in[\varepsilon,T]}
b_1(t) \leq \frac{C}{\rho\beta_{\mathcal{J}} \varepsilon} \leq \frac{C}{\rho r_A \varepsilon}
\end{equation}
for some universal constant $C$, using $r_A\leq \beta_{\mathcal{J}}$.

We now control $b_2$. Because $\mathbf{F}_t$ is of order $L$ we can write
\begin{equation}\label{eq:b2control}
b_2(t) \leq CL \rho \beta_A\int_0^t e^{-\rho \beta_A(t-u)} \left( \int_0^u e^{-\rho \beta_{\mathcal{J}}(u-v)}
    \dd v\right)\dd u
\quad \leq \frac{CL}{\rho \beta_{\mathcal{J}}}     \leq \frac{CL}{\rho r_A}
\end{equation}
for $C$ independent of $(\varepsilon,A,L)$.
Finally, we turn to $b_3$. Because the martingale $\mathbf{M}_u$ has uniformly bouded quadratic variation, we can use Lemma \ref{lem:controlOU} to write
\begin{align*}
\EE\left[\sup_{t\in[0,T]}b_3(t)\right]
    \leq& \rho \beta_A\int_0^te^{-\rho\beta_A(t-u)}\dd u 
    \EE\left[\sup_{t\in[0,T]}\left| \int_0^u e^{-\rho \beta_{\mathcal{J}}(u-v)}
    \left<\hat{\mathbf{w}}_{\mathcal{J}}^A,\dd \mathbf{M}_v\right>\right|\right]
    \leq \sqrt{\frac{1+\ln_+(\rho \beta_{\mathcal{J}}2T)}{\rho \beta_{\mathcal{J}}}}\\
    \leq& C\sqrt{\frac{\ln(\rho)}{\rho r_A}}
\end{align*}
using (\ref{eq:orderrAbeta1}) and $\rho\geq e$.
We obtain the result by combining this with (\ref{eq:b1control}) and (\ref{eq:b2control}) in (\ref{eq:b1b2b3}).
\end{proof}

\subsection{Summing controls of linkage equilibrium across loci}
We now prove the following estimates
\begin{prop}\label{prop:supDeltaS}
Recall from (\ref{def:epsilon})
\begin{gather*}
\varepsilon_L:=\frac{1}{\sqrt{\rho r^{**}}}
\end{gather*}
Let us assume (\ref{ass:rhobeta1}), that is,
\begin{gather*}
\rho r^{**} \gg L^{2}\ln(\rho) 
\end{gather*}
We have
\begin{align}
\lim_{L\to+\infty}\sum\limits_{\ell\in[L]}
\EE\left[\int_0^T |S^{\ell}(\mathbf{X}_t)- S^{\ell}(\pi(\mathbf{X}_t))|\dd t\right] 
=& 0
\label{eq:supDeltaS}\\
\lim_{L\to+\infty}
\frac{1}{L}\sum\limits_{\substack{\ell_1,\ell_2\in [L]\\\ell_1\neq \ell_2}}
\EE\left[
\sup_{t\in [\varepsilon_L,T]} \left|D^{\ell_1,\ell_2}(\mathbf{X}_t)\right|\dd t\right] =& 0
\label{eq:supDell}
\end{align}
Consider a sequence integers $\ell^L\equiv \ell\in[L]$ satisfying Assumptions (\ref{ass:rhorell}).
Then
\begin{align}\label{eq:supDeltaSell0}
\lim_{L\to+\infty}
L\EE\left[\int_0^T |S^{\ell}(\mathbf{X}_t)- S^{\ell}(\pi(\mathbf{X}_t))|\dd t\right] 
=& 0 
\end{align}
Finally, consider $\ell_1^L\equiv \ell_1,\ell_2^L\equiv \ell_2$ such that
\begin{equation}
\rho r_{\{\ell_1,\ell_2\}} \gg L\label{eq:rhorell1ell2}
\end{equation}
Then
\begin{equation}\label{eq:Dell1ell2}
\lim_{L_\to \infty}
\EE\left[\int_0^T |D^{\ell_1,\ell_2}(\mathbf{X}_t)|\right] = 0
\end{equation}
\end{prop}
We use the following computational Lemma.
\begin{lem}\label{lem:domsumA}
We can find a constant $C$ independent of $(L,\varepsilon)$ such that for any $\ell_0\in [L]$
\begin{multline}
\sum\limits_{\substack{A\subseteq [L]\smallsetminus \{\ell_0\}\\ 1\leq \#A\leq 2}}\frac{1}{L^{\#A}}\left(
\frac{1}{\rho r_{\{\ell_0\}\cup A}\varepsilon} 
+\frac{L}{\rho r_{\{\ell_0\}\cup A}} + \sqrt{\frac{\ln(\rho)}{\rho r_{\{\ell_0\}\cup A}}}
\right)\\
\leq
C\left(\left(\frac{1}{\rho \varepsilon} + \frac{L}{\rho}\right)
\left(\frac{1}{r_{\ell_0}^*} + \frac{1}{r^{**}}\right)
+ \sqrt{\frac{\ln(\rho)}{\rho r_{\ell_0}^*}} + \sqrt{\frac{\ln(\rho)}{\rho r^{**}}}\right)
\label{eq:sumrell}
\end{multline}
Furthermore,
\begin{equation}
\sum\limits_{\substack{A\subseteq [L]\\ 2\leq \#A\leq 3}}
\frac{1}{L^{\#A}}\left(
\frac{1}{\rho r_{A}\varepsilon} 
+\frac{L}{\rho r_{A}} + \sqrt{\frac{\ln(\rho)}{\rho r_{A}}}
\right)
\leq
C\left(\frac{1}{\rho r^{**} \varepsilon} + \frac{L}{\rho r^{**}}
 + \sqrt{\frac{\ln(\rho)}{\rho r^{**}}}\right)
\label{eq:sumrA}
\end{equation}
\end{lem}
\begin{proof}
In the follwing, we write $a_{\ell_0}$ for the left-hand side of (\ref{eq:sumrell}). Then 
\begin{align*}
a_{\ell_0}
=&
\left(\frac{1}{\rho \varepsilon} + \frac{L}{\rho}\right)
\sum\limits_{\substack{
        A\subseteq [L]\smallsetminus\{\ell_0\}\\ 1\leq \#A\leq 2}}
 \frac{1}{L^{\#A}r_{\{\ell_0\}\cup A}}
+ 
\sqrt{\frac{\ln(\rho)}{\rho}}
\sum\limits_{\substack{
        A\subseteq [L]\smallsetminus\{\ell_0\}\\ 1\leq \#A\leq 2}}
\frac{1}{L^{\#A}\sqrt{r_{\{\ell_0\}\cup A}}}
\end{align*}
From Jensen's inequality, 
\begin{align*}
\sum\limits_{\substack{
        A\subseteq [L]\smallsetminus\{\ell_0\}\\ 1\leq \#A\leq 2}}\frac{1}{K L^{\#A}}
\times \frac{1}{\sqrt{ r_{\{\ell_0\}\cup A}}}
\leq& 
\left(
\sum\limits_{\substack{
        A\subseteq [L]\smallsetminus\{\ell_0\}\\ 1\leq \#A\leq 2}}\frac{1}{K L^{\#A}}
\times \frac{1}{ r_{\{\ell_0\}\cup A}}\right)^{\frac{1}{2}}
\end{align*}
where $K$ is a normalization constant
$$K := 
\sum\limits_{\substack{A\subseteq [L]\smallsetminus\{\ell_0\}\\ 1\leq \#A\leq 2}}
\frac{1}{L^{\#A}}$$
Because $K$ is of order 1, we can absorb it into a universal constant $C$ and we get
\begin{equation}\label{eq:tmpsupdeltaS}
a_{\ell_0}
 \leq 
C\left(\left(\frac{1}{\rho \varepsilon} + \frac{L}{\rho}\right)H^{\ell_0}  + \sqrt{\ln(\rho)\frac{H^{\ell_0}}{\rho}}\right)
\end{equation}
where
$$
H^{\ell_0} :=
\sum\limits_{\substack{
        A\subseteq [L]\smallsetminus\{\ell_0\}\\ 1\leq \#A\leq 2}}\frac{1}{L^{\#A}}
\times \frac{1}{r_{\{\ell_0\}\cup A}}
$$
We write
\begin{align*}
H^{\ell_0}
=& \frac{1}{L} \sum\limits_{\ell_1\in [L]\smallsetminus \{\ell_0\}} 
 \frac{1}{ r_{\{\ell_0,\ell_1\}}}
\;+\;  \frac{1}{L^2}\sum\limits_{\substack{\ell_1,\ell_2\in [L]\smallsetminus \{\ell_0\}\\\ell_1\neq \ell_2}}
 \frac{1}{ r_{\{\ell_0,\ell_1,\ell_2\}}}\\
\leq& \frac{1}{ r_{\ell_0}^*}
\;+\;  \frac{1}{L^2}\sum\limits_{\substack{\ell_1,\ell_2\in [L]\smallsetminus \{\ell_0\}\\\ell_1\neq \ell_2}}
\frac{1}{ r_{\{\ell_0,\ell_1,\ell_2\}}}
\end{align*} 
from the definition of $r_{\ell_0}^*$ in (\ref{eq:defrell}) .

Since $r_A = \min\limits_{\ell_1\neq \ell_2 \in A} r_{\{\ell_1,\ell_2\}}$
$$
\frac{1}{r_{\{\ell_0,\ell_1,\ell_2\}}}\leq
\frac{1}{r_{\{\ell_0,\ell_1\}}}+
\frac{1}{r_{\{\ell_0,\ell_2\}}}+
\frac{1}{r_{\{\ell_1,\ell_2\}}}
$$
It follows
$$    \frac{1}{L^2}\sum\limits_{\substack{\ell_1,\ell_2\in [L]\smallsetminus \{\ell_0\}\\\ell_1\neq \ell_2}}
\frac{1}{r_{\{\ell_0,\ell_1,\ell_2\}}}\\
\leq
\frac{1}{L^2}\sum\limits_{\substack{\ell_1,\ell_2\in [L]\smallsetminus \{\ell_0\}\\\ell_1\neq \ell_2}}
\frac{1}{r_{\{\ell_1,\ell_2\}}}
+
\frac{1}{L^2}\sum\limits_{\substack{\ell_1,\ell_2\in [L]\smallsetminus \{\ell_0\}\\\ell_1\neq \ell_2}}
 \frac{1}{r_{\{\ell_0,\ell_1\}}}
+ \frac{1}{r_{\{\ell_0,\ell_2\}}}
$$
The first sum is smaller than $\frac{1}{r^{**}}$, the second one is smaller than $\frac{2}{r_{\ell_0}^*}$.
Putting it all together we find a constant $C$ independent of $(\varepsilon,L,\ell_0)$ such that
$$
H^{\ell_0} \leq C\left(\frac{1}{r_{\ell_0}^*} + \frac{1}{r^{**}}\right)
$$
We thus obtain that for any $\ell_0\in [L]$
$$
a_{\ell_0} \leq
C\left(\left(\frac{1}{\rho \varepsilon} + \frac{L}{\rho}\right)
\left(\frac{1}{r_{\ell_0}^*} + \frac{1}{r^{**}}\right)
+ \sqrt{\frac{\ln(\rho)}{\rho}}
\left(\frac{1}{r_{\ell_0}^*} + \frac{1}{r^{**}}\right)^{\frac{1}{2}}\right)
$$
for $C$ independent of $(\varepsilon,L,\ell_0)$. To get (\ref{eq:sumrell}), use that $\sqrt{x+y}\leq \sqrt{x}+\sqrt{y}$ for any $x,y\geq 0$. We turn to (\ref{eq:sumrA}). Write $a^*$ for the left-hand side of (\ref{eq:sumrA}). We have
\begin{align*}
a^* =& \frac{1}{L}\sum\limits_{\ell\in[L]} a_\ell\\
\leq& C\left(\left(\frac{1}{\rho \varepsilon} + \frac{L}{\rho}\right)
\left(\frac{1}{L}\sum\limits_{\ell\in[L]} \frac{1}{r_\ell^*}
+\frac{1}{r^{**}}\right)
+ \sqrt{\frac{\ln(\rho)}{\rho r^{**}}} 
+ \sqrt{\frac{\ln(\rho)}{\rho}}\frac{1}{L}\sum\limits_{\ell\in[L]}\frac{1}{\sqrt{r_\ell^*}}
\right)\\
\leq& C\left(\left(\frac{1}{\rho \varepsilon} + \frac{L}{\rho}\right)
\frac{2}{r^{**}}
+ \sqrt{\frac{\ln(\rho)}{\rho r^{**}}} 
+ \sqrt{\frac{\ln(\rho)}{\rho}}\frac{1}{L}\sum\limits_{\ell\in[L]}\frac{1}{\sqrt{r_\ell^*}}
\right)
\end{align*}
We conclude with Jensen's inequality
$$
\frac{1}{L}\sum\limits_{\ell\in[L]}\frac{1}{\sqrt{r_\ell^*}}
\leq \left(\frac{1}{L}\sum\limits_{\ell\in[L]}\frac{1}{r_\ell^*}\right)^{\frac{1}{2}} = \frac{1}{\sqrt{r^{**}}}
$$
\end{proof}

\begin{proof}[Proof of Proposition \ref{prop:supDeltaS}]
First, let us notice that (\ref{ass:rhobeta1}) implies
$$
\rho \gg \frac{L^{2}}{r^{**}} = \frac{L^{2}}{L(L-1)}\sum\limits_{\substack{\ell_1,\ell_2\in [L]\\\ell_1\neq\ell_2}} \frac{1}{r_{\{\ell_1,\ell_2\}}}
\geq \max_{\substack{\ell_1,\ell_2\in [L]\\\ell_1\neq\ell_2}} \frac{1}{r_{\{\ell_1,\ell_2\}}}
= \frac{1}{r_{[L]}}
$$
In particular, for any $A\subseteq [L]$ we have $\rho r_A \geq \rho r_{[L]} \gg 1$ so that we may apply Proposition \ref{prop:controllink} and Lemma \ref{lem:domsumA}. Second, we note that by (\ref{ass:rhobeta1}), $\varepsilon_L$ satisfies
\begin{equation}\label{eq:epsLsmall}
\frac{1}{\rho r^{**}} \ll \varepsilon_L    \ll \frac{1}{L}
\end{equation}

Let us start with (\ref{eq:supDeltaS}). We have
\begin{align*}
\sum\limits_{\ell\in[L]}
\EE\left[\int_0^T |S^{\ell}(\mathbf{X}_t)- S^{\ell}(\pi(\mathbf{X}_t))|\dd t\right]
=& \mathcal{O}(L\varepsilon_L) + \EE\left[\sum\limits_{\ell\in[L]}\int_{\varepsilon_L}^T |S^{\ell}(\mathbf{X}_t)- S^{\ell}(\pi(\mathbf{X}_t))|\dd t\right]\\
\leq& o(1) +  T\sum\limits_{\ell\in[L]} \EE\left[\sup\limits_{t\in[\varepsilon_L,T]}
 \left|S^{\ell}(\mathbf{X}_t)- S^{\ell}(\pi(\mathbf{X}_t))\right|\right]
\end{align*}
from (\ref{eq:epsLsmall}).
Using (\ref{eq:boundDSwithL2v2}) from Proposition \ref{prop:boundDSwithL2} we get
$$
\sum\limits_{\ell\in[L]} \EE\left[\sup\limits_{t\in[\varepsilon_L,T]}
 \left|S^{\ell}(\mathbf{X}_t)- S^{\ell}(\pi(\mathbf{X}_t))\right|\right]
 \leq CL\sum\limits_{\substack{A\subseteq [L]\\ 2\leq \#A\leq 3}} \frac{1}{L^{\#A}}
 \EE\left[\sup\limits_{t\in [\varepsilon_L,T]}||\mathbf{X}_t^A - \pi(\mathbf{X}_t^A)||_2\right]
$$
Then, Proposition \ref{prop:controllink} and (\ref{eq:sumrA}) from Lemma \ref{lem:domsumA} yield
$$
\sum\limits_{\ell\in[L]}
\EE\left[\sup\limits_{t\in[\varepsilon_L,T]}
 \left|S^{\ell}(\mathbf{X}_t)- S^{\ell}(\pi(\mathbf{X}_t))\right|\right]
\leq CL\left(\frac{1}{\rho r^{**} \varepsilon_L} + \frac{L}{\rho r^{**}}
 + \sqrt{\frac{\ln(\rho)}{\rho r^{**}}}\right)
$$
By (\ref{ass:rhobeta1}) and (\ref{eq:epsLsmall}), we obtain that the right-hand side is $o(1)$. We obtain (\ref{eq:supDell}) the same way, using (\ref{eq:boundDSwithL2v1}) from Proposition \ref{prop:boundDSwithL2}. 

We now consider a sequence $\ell^L\equiv \ell$ satisfying (\ref{ass:rhorell}). We define
$$
\varepsilon_{\ell,L}\equiv \varepsilon_\ell := \frac{1}{\sqrt{\rho r_\ell^*}} \vee \varepsilon_L
$$
Just as in (\ref{eq:epsLsmall}) we have by (\ref{ass:rhorell})
$$
\frac{1}{\rho r_\ell^* }\ll \varepsilon_\ell \ll \frac{1}{L}
$$
It follows as before
\begin{align*}
L\EE\left[\int_0^T |S^{\ell}(\mathbf{X}_t)- S^{\ell}(\pi(\mathbf{X}_t))|\dd t\right]
\leq& \mathcal{O}(L\varepsilon_\ell)
+ TL\EE\left[\sup\limits_{t\in[\varepsilon_\ell,T]}
\left|S^{\ell}(\mathbf{X}_t)- S^{\ell}(\pi(\mathbf{X}_t))\right|\right] \\
\leq& o(1)
+ TL\EE\left[\sup\limits_{t\in[\varepsilon_\ell,T]}
\left|S^{\ell}(\mathbf{X}_t)- S^{\ell}(\pi(\mathbf{X}_t))\right|\right] \\
\end{align*}
Using Proposition \ref{prop:controllink} and (\ref{eq:sumrell}) from Lemma \ref{lem:domsumA} we get
\begin{align*}
TL\EE\left[\sup\limits_{t\in[\varepsilon_\ell,T]}
\left|S^{\ell}(\mathbf{X}_t)- S^{\ell}(\pi(\mathbf{X}_t))\right|\right] 
\leq& CL\left(\left(\frac{1}{\rho \varepsilon_L} + \frac{L}{\rho}\right)
\left(\frac{1}{r_{\ell}^*} + \frac{1}{r^{**}}\right)
+ \sqrt{\frac{\ln(\rho)}{\rho r_{\ell}^*}} + \sqrt{\frac{\ln(\rho)}{\rho r^{**}}}\right)\\
\leq& C\left(\frac{L}{\rho r_{\ell}^* \varepsilon_L} + \frac{L}{\rho r_\ell^*} + 
\frac{L}{\rho r^{**} \varepsilon_L} + \frac{L}{\rho r^{**}} 
+ \sqrt{\frac{\ln(\rho) L^2}{\rho r_{\ell}^*}} + \sqrt{\frac{\ln(\rho) L^2}{\rho r^{**}}}\right)
\end{align*}
for some constant $C$ independent of $(L,\ell)$. The right-hand side is small from (\ref{ass:rhobeta1}),(\ref{ass:rhorell}) and the definition of $\varepsilon_\ell$.

Finally, to prove (\ref{eq:Dell1ell2}) we define $\varepsilon_{\ell_1,\ell_2,L}\equiv \varepsilon_{\ell_1,\ell_2} := \frac{1}{\sqrt{\rho r_{\{\ell_1,\ell_2\}}}}\vee\frac{1}{\sqrt{\rho r^{**}}}$. We have
$$
\int_0^t \left|D^{\ell_1,\ell_2}(\mathbf{X}_u)\right|\dd u
\leq \varepsilon_{\ell_1,\ell_2} + t\sup\limits_{u\in [\varepsilon_{\ell_1,\ell_2},T]} \left|D^{\ell_1,\ell_2}(\mathbf{X}_t)\right|
$$
The first term on the right-hand side is $o(1)$ from Assumption (\ref{eq:rhorell1ell2}). For the second term, (\ref{eq:boundDSwithL2v1}) from Proposition \ref{prop:boundDSwithL2} and Proposition \ref{prop:controllink} tell us
$$
\EE\left[\sup\limits_{u\in [\varepsilon_{\ell_1,\ell_2},T]} |D^{\ell_1,\ell_2}(\mathbf{X}_t)|\right]
\leq C\left( \frac{1}{\rho r_{\{\ell_1,\ell_2\}}\varepsilon_{\ell_1,\ell_2}}
+\frac{L}{\rho r_{\{\ell_1,\ell_2\}}} + \sqrt{\frac{\ln(\rho)}{\rho r_{\{\ell_1,\ell_2\}}}}\right)
$$
The result follows from Assumption (\ref{eq:rhorell1ell2}) and the definition of $\varepsilon_{\ell_1,\ell_2}$.
\end{proof}

\subsection{Proof of Theorem \ref{thm:strongcvgrho} \label{sec:cclstrongrec}}
In this section, we make the following assumptions of Theorem \ref{thm:strongcvgrho}:
\begin{itemize}
    \item $\mu_{\mathbf{X}_0}$ converges in law to a deterministic measure $m_0$.
    \item We have $\rho r^{**} \gg L^{2}\ln(\rho)\qquad$ (\ref{ass:rhobeta1})
\end{itemize}
The proof is broken down into several lemmas. 
Let us define for $\mathbf{y}\in\XX^{[L]}$ and $\varphi\in\mathcal{C}^2([0,1])$
\begin{equation}\label{dfn:Gell}
    G^\ell_{\mathbf{y}}\varphi(x) := x(1-x)LS^\ell(\mathbf{y})\varphi'(x) + \overline{\Theta}(x)\varphi'(x) + \frac{1}{2}x(1-x)\varphi''(x)
\end{equation}
$G^\ell_{\mathbf{X}_t}$ can be thought of as the generator for the $\ell-$th locus.
And the limit generator is
\begin{equation}\label{dfn:barG}
\forall x\in [0,1],\xi\in \mathcal{P}([0,1]),\qquad 
\overline{G}_{\xi}\varphi(x) := \left(\overline{s}(\xi)x(1-x)+\overline{\Theta}(x)\right)\varphi'(x) + \frac{1}{2}x(1-x)\varphi''(x)
\end{equation}

\begin{lem}\label{lem:GGbarto0}
Consider a sequence $\ell^L\equiv \ell \in [L]$ satisfying (\ref{ass:rhorell}), that is
$$
    \rho r_{\ell}^* \gg L^2\ln(\rho)
$$
We have for $\varphi \in \mathcal{C}^2([0,1])$
\begin{equation}
\EE\left[\int_{0}^{T} \dd u\;\left| G^{\ell}_{\mathbf{X}_u}\varphi(p^{\ell}(\mathbf{X}_u)) - \overline{G}_{\mu_{\mathbf{X}_u}}\varphi(p^{\ell}(\mathbf{X}_u))\right|\right] \longrightarrow 0
\label{eq:GbarGelltozero}
\end{equation}

Furthermore, for any $\varphi\in \mathcal{C}^2([0,1])$ we have
\begin{equation}
\EE\left[\frac{1}{L}\sum\limits_{\ell\in [L]}
\int_{0}^T \dd u\;\left| G^\ell_{\mathbf{X}_u}\varphi(p^\ell(\mathbf{X}_u)) - \overline{G}_{\mu_{\mathbf{X}_u}}\varphi(p^\ell(\mathbf{X}_u))\right|\right] \longrightarrow 0
\label{eq:GbarGmeantozero}
\end{equation}
\end{lem}
\begin{proof}
We first consider the case where $\ell^L\equiv \ell$ is such that $\rho r_{\ell}^* \gg L^{2}\ln(\rho)$. For any $t\in [0,T]$ we have from (\ref{dfn:Gell})
\begin{align*}
\left|G^{\ell}_{\mathbf{X}_t}\varphi(p^{\ell}(\mathbf{X}_t)) - \overline{G}_{\mu_{\mathbf{X}_t}}\varphi(p^{\ell}(\mathbf{X}_t))\right|
=& \left|LS^{\ell}(\mathbf{X}_t) - \overline{s}(\mu_{\mathbf{X}_t})p^{\ell}(\mathbf{X}_t)(1-p^{\ell}(\mathbf{X}_t)) \right|\times|\varphi'(\mathbf{X}_t)| 
\end{align*}
Observe that 
\begin{multline*}
    \left|LS^{\ell}(\mathbf{X}_t) - \overline{s}(\mu_{\mathbf{X}_t})p^{\ell}(\mathbf{X}_t)(1-p^{\ell}(\mathbf{X}_t)) \right|
\leq L\left|S^{\ell}(\mathbf{X}_t) - S^{\ell}(\pi(\mathbf{X}_t))\right|\\
+\left|LS^{\ell}(\pi(\mathbf{X}_t))- \overline{s}(\mu_{\mathbf{X}_t})p^{\ell}(\mathbf{X}_t)(1-p^{\ell}(\mathbf{X}_t)) \right|
\end{multline*}
The expectation of the integral of the first term is $o(1)$ from (\ref{eq:supDeltaSell0}) in Proposition \ref{prop:supDeltaS}.
The expectation of the second term is $\mathcal{O}\left(\frac{1}{L}\right)$ from Lemma \ref{lem:pibarom}. To prove (\ref{eq:GbarGmeantozero}), 
\begin{multline*}
\left|G^{\ell}_{\mathbf{X}_t}\varphi(p^{\ell}(\mathbf{X}_t)) - \overline{G}_{\mu_{\mathbf{X}_t}}\varphi(p^{\ell}(\mathbf{X}_t))\right|
= \frac{1}{L}\sum\limits_{\ell\in[L]}\left|LS^{\ell}(\mathbf{X}_t) - \overline{s}(\mu_{\mathbf{X}_t})p^{\ell}(\mathbf{X}_t)(1-p^{\ell}(\mathbf{X}_t)) \right|\times|\varphi'(\mathbf{X}_t)| \\
\leq
\frac{||\varphi'||_\infty}{L}\sum\limits_{\ell\in[L]}L
\left|S^{\ell}(\mathbf{X}_t) - S^{\ell}(\pi(\mathbf{X}_t))\right|\\
+\frac{||\varphi'||_\infty}{L}\sum\limits_{\ell\in[L]}
\left|LS^{\ell}(\pi(\mathbf{X}_t))- \overline{s}(\mu_{\mathbf{X}_t})p^{\ell}(\mathbf{X}_t)(1-p^{\ell}(\mathbf{X}_t)) \right|
\end{multline*}
The expectation of the integral of the first term is $o(1)$ from (\ref{eq:supDeltaS}) in Proposition \ref{prop:supDeltaS}, and the second term is small from Lemma \ref{lem:pibarom}.
\end{proof}
To get the convergence of $(\mu_{\mathbf{X}_t})_{t\in[0,T]}$ we follow a classical proof, with first a tightness Lemma, then proof that we get the correct martingale problem in the limit.
\begin{lem}\label{lem:tightness}
        Consider a sequence $\ell^L\equiv \ell \in [L]$ such that (\ref{ass:rhorell}) is satisfied, that is,
        \begin{gather*}
                \rho r_{\ell}^* \gg L^{2}\ln(\rho)
        \end{gather*}
        The law of $(p^{\ell}(\mathbf{X}_t))_{t\in [0,T]}\; \in \DD([0,T],[0,1])$ is tight for the Skorokhod J1 topology. Furthermore, the allelic law process $(\mu_{\mathbf{X}_t})_{t\in [0,T]} \; \in \DD([0,T],\mathcal{P}([0,1]))$ is tight for the Skorokhod J1 topology.
\end{lem}
\begin{proof}
        To prove tightness of $p^\ell(X_t)_{t\in[0,T]}$ in $\DD([0,T],[0,1])$ for the Skorokhod J1 topology, we use the classical Rebolledo criterion (see Theorem C.4 in \cite{chaintron2021propagation2}). We first separate martingale and finite variation part from the SDE (\ref{eq:SDEp(X)}).
    $$
    p^{\ell}(\mathbf{X}_{t}) - p^{\ell}(\mathbf{X}_0) = \int_{0}^{t} G_{\mathbf{X}_u}^{\ell}\mbox{Id}(p^{\ell}(\mathbf{X}_u)))\dd u +
    \int_{0}^{t} \sqrt{p^{\ell}(\mathbf{X}_u)(1-p^{\ell}(\mathbf{X}_u))}
    \dd\hat{B}_u^{\ell} =: V^{\ell}_{t} + M^{\ell}_{t}
    $$
    The Rebolledo criterion has three conditions:
    \begin{enumerate}
        \item $p^{\ell}(\mathbf{X}_0)$ is tight.
        \item $V^{\ell}_t$ and $M^{\ell}_t$ are tight for all $t\in [0,T]$.
        \item For $A=V^{\ell}$ or $A=\braket{M^{\ell}}^{QV}$ we have
$$\forall \varepsilon >0,\qquad 
    \lim\limits_{\delta \downarrow 0} \limsup\limits_{L\to +\infty}\sup\limits_{(t_1,t_2)\in\mathcal{S}_\delta}
    \PP\left[ |A_{t_2}-A_{t_1}| >\varepsilon \right] 
    = 0$$
where $\mathcal{S}_\delta$ is the set of pairs of stopping times $(t_1,t_2)$ for the natural filtration of the process $(\mathscr{F}_t)_{t\in [0,T]}$, such that $|t_1-t_2|<\delta$ a.s.
    \end{enumerate}
    Since $p^{\ell}(\mathbf{X}_0)$ lives in $[0,1]$, the first condition is trivially satisfied.
    We turn  to the second condition. Since $\braket{M^{\ell}}_t^{QV}\leq t$ from Corollary \ref{cor:SDEp(X)}, $M^{\ell}_t$ is tight. Furthermore, 
 since $V^{\ell}_t = p^{\ell}(\mathbf{X}_t) - p^{\ell}(\mathbf{X}_0) - M^{\ell}_t$, it is necessarily tight. 
The third condition for $A=\braket{M^{\ell}}^{QV}$ is trivially satisfied for $\braket{M^{\ell}}^{QV}$ since for $t\in [0,T]$ we have $\frac{\dd}{\dd t}\braket{M^\ell}_t \leq 1$. 
For $A=V^{\ell}$, we write
$$
\int_{t_1}^{t_2} |G_{\mathbf{X}_u}^{\ell}\mbox{Id}(p^{\ell}(\mathbf{X}_u)))|\dd u
\leq
\int_{t_1}^{t_2} |\overline{G}_{\mu_{\mathbf{X}_u}}\mbox{Id}(p^{\ell}(\mathbf{X}_u)))|\dd u
+
\int_{0}^T \left|G_{\mathbf{X}_u}^{\ell}\mbox{Id}(p^{\ell}(\mathbf{X}_u)))
- \overline{G}_{\mu_{\mathbf{X}_u}} \mbox{Id}(p^{\ell}(\mathbf{X}_u)))\right|\dd u
$$
The first term can be bounded by $C|t_1-t_2|$ for some deterministic constant $C$, the second term goes to 0 in probability from (\ref{eq:GbarGelltozero}) in Lemma \ref{lem:GGbarto0} unifrmly in $t_1,t_2$.
This yields the tightness of $(p^{\ell}(\mathbf{X}_t))_{t\in[0,T]}$.

To get the tightness of $(\mu_{\mathbf{X}_t})_{t\in [0,T]}\in \DD([0,T],\mathcal{P}([0,1]))$, following Theorem 2.1 of \cite{roelly1986criterion}, we only need to show that for any $\varphi\in \mathcal{C}^2([0,1])$, the process $(<\mu_{\mathbf{X}_t},\varphi>)_{t\in [0,T]}$ is tight. We may again apply the Rebolledo criterion, writing
$$<\mu_{\mathbf{X}_t},\varphi>
- <\mu_{\mathbf{X}_0},\varphi>\quad
= \frac{1}{L} \sum\limits_{\ell\in [L]} \varphi(p^\ell(\mathbf{X}_t)) - \varphi(p^\ell(\mathbf{X}_0))
= \frac{1}{L} \sum\limits_{\ell\in [L]}  V^{\varphi,\ell}_t + M^{\varphi,\ell}_t$$
with
$$
V_t^{\varphi,\ell} := \int_0^t G_{\mathbf{X}_u}^\ell \varphi (p^\ell(\mathbf{X}_u)) \dd u
\qquad ;\qquad
M_t^{\varphi,\ell} := \int_{0}^{t}\varphi'(p^\ell(\mathbf{X}_u)) \sqrt{p^{\ell}(\mathbf{X}_u)(1-p^{\ell}(\mathbf{X}_u))}
    \dd\hat{B}_u^{\ell}
$$
The condition 1 of the Rebolledo criterion is easily verified, since $<\mu_{\mathbf{X}_0},\varphi>$ is tight because bounded by $||\varphi||_\infty$. The condition 2 is obtained as above. Indeed,  on the one hand, the quadratic variation of $M_t^{\varphi,\ell}$ is uniformly bounded by $T||\varphi'||_\infty$ so $M_t^{\varphi,\ell}$ is tight for any fixed $t\in[0,T]$. On the other hand, since $p^\ell(\mathbf{X}_t)$ is in $[0,1]$, it is also tight, and therefore $V_t^{\varphi,\ell}$ is tight.

Controlling $L^{-1} \sum\limits_{\ell\in [L]}V_t^{\varphi,\ell}$ works just as above, using (\ref{eq:GbarGmeantozero}) instead of (\ref{eq:GbarGelltozero}). We only need to control the martingale $L^{-1}\sum\limits_{\ell\in [L]} M_t^{\varphi,\ell}$. We have for $t_1,t_2\in [0,T]$
$$\left<\frac{1}{L}\sum\limits_{\ell\in [L]} M^{\varphi,\ell}\right>_{t_2}^{QV}
- 
\left<\frac{1}{L}\sum\limits_{\ell\in [L]} M^{\varphi,\ell}\right>_{t_1}^{QV}
\quad = \quad \frac{1}{L^2}\sum\limits_{\ell_1, \ell_2\in [L]} \braket{M^{\varphi,\ell_1},M^{\varphi,\ell_2}}_{t_2}^{QV} - \braket{M^{\varphi,\ell_1},M^{\varphi,\ell_2}}_{t_1}^{QV}$$
Using the Kunita-Watanabe inequality (see Corollary 1.16, chapter IV in \cite{revuz2013continuous}), we have
\begin{align*}
\left|\braket{M^{\varphi,\ell_1},M^{\varphi,\ell_2}}_{t_2}^{QV} - \braket{M^{\varphi,\ell_1},M^{\varphi,\ell_2}}_{t_1}^{QV}\right|
\leq& \sqrt{\left|\braket{M^{\varphi,\ell_1}}_{t_2}^{QV}-\braket{M^{\varphi,\ell_1}}_{t_1}^{QV}\right|\times\left|\braket{M^{\varphi,\ell_2}}_{t_2}^{QV}
-\braket{M^{\varphi,\ell_2}}_{t_1}^{QV}\right|}
\end{align*}
This is uniformly bounded by $||\varphi'||_\infty^2 \;|t_2-t_1|$, which yields the result.
\end{proof}

We can now show that $(\mu_{\mathbf{X}_t})_{t\in [0,T]}$ converges to a solution of (\ref{eq:McKean-limitW}).
\begin{proof}[Proof of Theorem \ref{thm:strongcvgrho} part 1]
We will prove for any function $\varphi\in \mathcal{C}^2([0,1])$
$$
<\mu_{\mathbf{X}_t},\varphi> - <\mu_{\mathbf{X}_0},\varphi> - \int_0^t <\mu_{\mathbf{X}_u},\overline{G}_{\mu_{\mathbf{X}_u}} \varphi> \dd u \longrightarrow 0
$$
in probability. Notice that
\begin{align*}
M_t^\varphi:=&
\frac{1}{L}\sum\limits_{\ell\in [L]}
\varphi(p^\ell(\mathbf{X}_t)) - \varphi(p^\ell(\mathbf{X}_0))
- \int_0^t \dd u  G^\ell_{\mathbf{X}_u} \varphi(p^{\ell}(\mathbf{X}_u))
\\
=& <\mu_{\mathbf{X}_t},\varphi> - <\mu_{\mathbf{X}_0},\varphi>
- \int_0^t \dd u \frac{1}{L}\sum\limits_{\ell\in [L]} G^\ell_{\mathbf{X}_u} \varphi(p^{\ell}(\mathbf{X}_u))
\end{align*}
is a martingale, which by It\^o's formula and  Corollary \ref{cor:SDEp(X)} satisfies
$$
\dd M_t^\varphi := \frac{1}{L}\sum\limits_{\ell\in [L]}  \sqrt{p^\ell(\mathbf{X}_t)(1-p^\ell(\mathbf{X}_t))}\;\varphi'(p^\ell(\mathbf{X}_t))\dd\hat{B}_t^\ell.
$$
We split the equation between a martingale and non-martingale part
\begin{multline*}
\left|<\mu_{\mathbf{X}_t},\varphi> - <\mu_{\mathbf{X}_0},\varphi> - \int_0^t \dd u <\mu_{\mathbf{X}_u},\overline{G}_{\mu_{\mathbf{X}_u}} \varphi> \right| \\
= \left|M_t^\varphi\right| 
+ \int_0^t \dd u\left| <\mu_{\mathbf{X}_u},\overline{G}_{\mu_{\mathbf{X}_u}} \varphi> - \frac{1}{L}\sum\limits_{\ell\in [L]} G^\ell_{\mathbf{X}_u} \varphi(p^{\ell}(\mathbf{X}_u)) \right|
\end{multline*}
Let us show $M_t^\varphi$ goes to 0.
We have
\begin{multline*}
    \dd \braket{M^\varphi}_t^{QV} = \frac{1}{L^2}\sum\limits_{\ell\in [L]}\varphi'(p^\ell(\mathbf{X}_t))^2 p^\ell(\mathbf{X}_t)(1-p^\ell(\mathbf{X}_t))\dd t\\
+ \frac{1}{L^2}\sum\limits_{\substack{\ell_1, \ell_2\in [L]\\\ell_1\neq\ell_2}}
\varphi'(p^{\ell_1}(\mathbf{X}_t))\varphi'(p^{\ell_2}(\mathbf{X}_t))\;
\dd \braket{p^{\ell_1}(\mathbf{X}),p^{\ell_2}(\mathbf{X})}_t^{QV}\\
\leq \frac{1}{L}||\varphi'||_\infty^2
+ \frac{1}{L^2}\sum\limits_{\substack{\ell_1, \ell_2\in [L]\\\ell_1\neq\ell_2}}
||\varphi'||_\infty^2\;
 \left|D^{\ell_1,\ell_2}(\mathbf{X}_t)\right|\dd t
\end{multline*}
where we used (\ref{eq:covell1ell2}) from Corollary \ref{cor:SDEp(X)}.
It follows
$$
\braket{M^\varphi}_T^{QV} \leq \frac{1}{L}||\varphi'||_\infty^2 T
+ \frac{1}{L^2}\sum\limits_{\substack{\ell_1, \ell_2\in [L]\\\ell_1\neq\ell_2}}
||\varphi'||_\infty^2\;
\int_0^T \left|D^{\ell_1,\ell_2}(\mathbf{X}_t) \right|\dd t
$$
The first term goes to 0. The expectation of the second term is $o_{\mathbb{P}}(1)$ from (\ref{eq:supDell}) in Proposition \ref{prop:supDeltaS}. Since the quadratic variation goes to $0$, it follows that $M^\varphi$ goes to zero in the Skorokhod toplogy from the Dambis, Dubins-Schwartz theorem (Theorem 1.6, chapter V of \cite{revuz2013continuous}).

We now need to control
$$
\int_0^t \dd u \left|<\mu_{\mathbf{X}_u},\overline{G}_{\mu_{\mathbf{X}_u}} \varphi> - \frac{1}{L}\sum\limits_{\ell\in [L]} G^\ell_{\mathbf{X}_u} \varphi(p^{\ell}(\mathbf{X}_u)) \right|
\leq 
 \frac{1}{L}\sum\limits_{\ell\in [L]}
\int_0^T \dd u\left|\overline{G}_{\mu_{\mathbf{X}_u}} \varphi(p^\ell(\mathbf{X}_u)) -
G^\ell_{\mathbf{X}_u} \varphi(p^{\ell}(\mathbf{X}_u)) \right|
$$
The result follows from (\ref{eq:GbarGmeantozero}) from Lemma \ref{lem:GGbarto0}.

We thus obtain that any subsequential limit of $(\mu_{\mathbf{X}_t})_{t\in [0,T]}$ must satisfy equation (\ref{eq:McKean-limitW}), with initial law $m_0$. In particular, this yields existence of a solution to (\ref{eq:McKean-limitW}), which is unique from Proposition \ref{cor:uniqueness}. This completes the proof of the first part Theorem \ref{thm:strongcvgrho} 
\end{proof}

We conclude with the second part of the Theorem.
\begin{proof}[Proof of Theorem \ref{thm:strongcvgrho} part 2]
    First, notice that for every $i\in [n]$, Lemma \ref{lem:tightness} implies the tightness of $(p^{\ell_i^L}(\mathbf{X}))_{L\geq 1}$. Since a finite union of tight families is tight, we also get tightness for $(p^{\ell_i^L}(\mathbf{X}))_{L\geq n,i\in [n]}$. We will write $\ell_i^L\equiv \ell_i$ to alleviate notations.
    
For $\varphi\in\mathcal{C}^2([0,1]^n)$, $\mathbf{x}\in\XX^{[L]}$ and $\mathbf{y}\in[0,1]^n$ we define
$$
\mathcal{G}_{\mathbf{x}}\varphi(\mathbf{y}) := 
\sum\limits_{i\in[n]} G^{\ell_i}_{\mathbf{x}} \diamond_i 
\varphi(\mathbf{y})
+ \frac{1}{2}\sum\limits_{\substack{i,j\in[n]\\ i\neq j}}
\partial_{i,j} \varphi(\mathbf{y}) D^{\ell_i,\ell_j}(\mathbf{x})
$$
where $G^\ell_{\mathbf{x}} \diamond_i \varphi(\mathbf{y})$ means $G^\ell_{\mathbf{x}}$ applied to the function $z\mapsto \varphi(y^1,\dots,y^{i-1},z,y^{i+1},\dots,y^n)$, evaluated in $z=y^i$.
Define $\mathbf{Y}_t := (p^{\ell_i}(\mathbf{X}_t))_{i\in[n]}$. We know from It\^o's formula that
$$
M_t^\varphi := \varphi(\mathbf{Y}_t) - \varphi(\mathbf{Y}_0) - \int_0^t \dd u\;\mathcal{G}_{\mathbf{X}_u}\varphi(\mathbf{Y}_u)
$$
is a martingale for any $\varphi\in\mathcal{C}^2([0,1]^n)$. Furthermore, using $\braket{p(\mathbf{X})}_t\leq t$, we can find uniform bounds for the quadratic variation of $M_t^\varphi$. It remains to show
$$
\int_0^t \dd u\;\mathcal{G}_{\mathbf{X}_u}\varphi(\mathbf{Y}_u) - \int_0^t \dd u\sum\limits_{i\in[n]} \overline{G}_{\xi_u}\diamond_i\varphi(\mathbf{Y}_u) \longrightarrow 0
$$
where $\xi$ is the limit of $\mu_{\mathbf{X}}$. We write
\begin{multline*}
\left|\int_0^t\dd u \;\mathcal{G}_{\mathbf{X}_u}\varphi (\mathbf{Y}_u)  - \int_0^t\dd u \sum\limits_{i\in[n]} \overline{G}_{\xi_u}\diamond_i\varphi(\mathbf{Y}_u)\right|\\
\leq
\sum_{i\in[n]}
\int_0^t \dd u \left| G^{\ell_i}_{\mathbf{X}_u} \diamond_i \varphi(\mathbf{Y}_u)
-\overline{G}_{\mu_{\mathbf{X}_u}}\diamond_i\varphi(\mathbf{Y}_u)\right|
+ \sum_{i\in[n]}
\int_0^t \dd u \left|\overline{G}_{\mu_{\mathbf{X}_u}}\diamond_i\varphi(\mathbf{Y}_u)
-\overline{G}_{\xi_u}\diamond_i\varphi(\mathbf{Y}_u)\right|\\
+ \frac{1}{2}\sum\limits_{\substack{i,j\in[n]\\ i\neq j}}
\int_0^t\dd u\;\left|\partial_{i,j} \varphi(\mathbf{Y}_u) D^{\ell_i,\ell_j}(\mathbf{X}_u)\right|
\end{multline*}
The first term on the right-hand side goes to zero from (\ref{eq:GbarGelltozero}) in Lemma \ref{lem:GGbarto0}. The second term goes to zero from the first part of the theorem. For the third, we use (\ref{eq:Dell1ell2}) in Proposition \ref{prop:supDeltaS}.
\end{proof}

\subsection{Convergence of the genetic variance\label{sec:phenotype}}
Here, we show Theorem \ref{thm:geneticvariance}, which states that the population trait variance converges as $L\to +\infty$ to the genetic variance.
\begin{theorem8}
    Set $\varepsilon_L:=\frac{1}{\sqrt{\rho r^{**}}}$ and define the genetic variance $\sigma_t^2:=4\EE[f_t(1-f_t)]$ where $(f_t)_{t\in[0,T]}$ is solution to (\ref{eq:McKean-limitW}). Under the assumptions of Theorem \ref{thm:strongcvgrho} we have
    $$\EE\left[\sup\limits_{t\in [\varepsilon_L,T]}\left|L\mathbf{Var}_{\mathbf{X}_t}[Z(g)] - \sigma_t^2\right|\right]
    \longrightarrow 0$$
\end{theorem8}
\begin{proof}
First, note that
$$
\mathbf{Var}_{\mathbf{X}_t}[Z(g)]
= \frac{1}{L^2} \sum\limits_{\ell_1,\ell_2 \in [L]} \mathbf{Cov}_{\mathbf{X}_t}[g^{\ell_1},g^{\ell_2}]
$$
    This implies
\begin{align*}
    \big|L\mathbf{Var}_{\mathbf{X}_t}[Z(g)]-&4\EE[f_t(1-f_t)]\big|\\
    \leq& 
    \left|\frac{1}{L}\sum\limits_{\ell\in [L]} 
    \mathbf{Var}_{\mathbf{X}_t}[g^\ell] \ - \ 4\EE[f_t(1-f_t)]
    \right|\ 
    +\  \frac{1}{L}\sum\limits_{\substack{\ell_1, \ell_2\in [L]\\\ell_1\neq\ell_2}}
            \left|\mathbf{Cov}_{\mathbf{X}_t}[g^{\ell_1},g^{\ell_2}]\right|\\ 
    \leq& 
    \left|\frac{1}{L}\sum\limits_{\ell\in [L]} 
    4p^\ell(\mathbf{X}_t)(1-p^\ell(\mathbf{X}_t)) \ - \ 4\EE[f_t(1-f_t)]
    \right| +
    \frac{1}{L}\sum\limits_{\substack{\ell_1, \ell_2\in [L]\\\ell_1\neq\ell_2}} \left|D^{\ell_1,\ell_2}(\mathbf{X}_t)\right|\\
    \leq& \;
    4\left|<\mu_{\mathbf{X}_t},\mbox{Id}(1-\mbox{Id})>  -  \EE[f_t(1-f_t)]
    \right|\ +
    \frac{1}{L}\sum\limits_{\substack{\ell_1, \ell_2\in [L]\\\ell_1\neq\ell_2}} \left|D^{\ell_1,\ell_2}(\mathbf{X}_t)\right|
\end{align*}
where in the second inequality we used that under $\mathbf{X}_t$, $g^\ell$ has law $2\mbox{Ber}(p^\ell(\mathbf{X}_t))-1$. Taking the expectation on both sides of the inequality,
the convergence of the first term on the right-hand side follows from Theorem \ref{thm:strongcvgrho}. For the second term we use (\ref{eq:supDell}) in Proposition \ref{prop:supDeltaS}.
\end{proof}

\section{Acknowledgements}
We thank Louis-Pierre Chaintron for fruitful discussions and helpful references. Special thanks to Elisa Couvert for being the needed interlocutor in times of despair. We thank Reinhard B\"urger for helpful reference. 

Thanks also go to Guy Alarcon whose teachings motivated Courau to study mathematics. 

For funding, Philibert Courau and Amaury Lambert thank the Center for Interdisciplinary Research in Biology (CIRB, Coll\`ege de France) and the Institute of Biology of ENS (IBENS, \'Ecole Normale Sup\'erieure, Universit\'e PSL).

E.S. gratefully acknowledge support from the FWF project PAT3816823.
\addcontentsline{toc}{section}{References}
{\setlength{\baselineskip}{.8\baselineskip}
\bibliographystyle{unsrt}
\bibliography{biblio.bib}
\par}

\appendix
\section{Appendix}
\subsection{The  \texorpdfstring{$L_1$}{TEXT} law of the iterated logarithm \label{sec:apLIL}}
\begin{thm}
There is a universal constant $C>0$ such that for any a continuous martingale $(M_t)_{t\in [0,T)}$ with $0\leq T\leq +\infty$ we have
$$
\EE\left[\sup\limits_{t\in [0,T)}\frac{M_t}{\sqrt{1+ \left(\braket{M}_t^{QV} \ln_{(2)}\left(\braket{M}_t^{QV}\right)\right)}}\right] \leq C
$$
\end{thm}
\begin{proof}
First, let us show that we can assume $M$ to be a Brownian motion and $T=+\infty$. Indeed, the Dambis, Dubins-Schwartz theorem (Theorem 1.6, chapter V of \cite{revuz2013continuous}) states that if we define $\tau_t := \braket{M}_t^{QV}$ and $\sigma_u:= \inf\{t\geq 0, \tau_t \geq u\}$, then $(M_{\sigma_u})_{0\leq u \leq \tau_T}$ is a Brownian motion on $[0,\tau_T]$, noted $(B_t)_{t\in [0,\tau_T]}$. Evidently
$$
\EE\left[\sup\limits_{t\in [0,T)}\frac{M_t}{\sqrt{1+ (\tau_t^{QV} \ln_{(2)}(\tau_t^{QV}))}}\right] = 
\EE\left[\sup\limits_{u\in [0,\tau_T)}\frac{B_u}{\sqrt{1+ (u \ln_{(2)}(\sqrt{u}}))}\right]
$$
Up to increasing the probability space, we may assume $B$ is well defined on all of $\RR_+$. Furthermore, as $u$ goes to infinity $\ln_{(2)}(\sqrt{u})$ and $\ln_{(2)}(u)$ are equivalent, so we can replace one by the other.

We know from the global law of the iterated logarithm (see \cite{levy1954mouvement}, p.13) that
$$
\PP\left[\limsup\limits_{u\to +\infty} \left|\frac{B_u}{\sqrt{2 u\ln_{(2)}(u)}}\right| = 1\right] = 1
$$

It follows that $\left(\frac{B_u}{\sqrt{1 + (2 u \ln_{(2)}(u)})}\right)_{t\geq 0}$ is a continuous Gaussian process, with mean 0 and asymptotically bounded by 1 as $u\to +\infty$. In particular its supremum (resp. infimum) is almost surely finite. To see that the supremum is of finite expectation, we use \cite{landau1970supremum}. In this paper, the authors show in (1.2) that for any sequence of (possibly correlated) jointly gaussian random variables $(X_n)_{n\in \NN}$, such that $\PP[\sup_n  |X_n| <+\infty] = 1$, we necessarily have $\EE[\sup_n |X_n|] <+\infty$ (even stronger, they show that we can find small enough $\varepsilon>0$ such that $\EE[e^{\varepsilon \sup_n |X_n|^2}]<+\infty$).  We can apply this to $\left(\frac{B_u}{\sqrt{1 + (2 u \ln_{(2)}(u)})}\right)_{u\in \QQ_+}$ and obtain the result.
\end{proof}

\end{document}